\documentclass[12pt]{amsart}

\usepackage{amsmath,amsfonts,amssymb,amsthm,mathabx,shuffle,latexsym,mathtools,mathrsfs}
\usepackage{enumitem}
\usepackage[usenames, dvipsnames, svgnames]{xcolor}
\usepackage{ytableau}
\usepackage{tikz-cd}
\usepackage{pbox}
\makeatletter
\@namedef{subjclassname@2020}{\textup{2020} Mathematics Subject Classification}
\makeatother

\newtheorem{theorem}{Theorem}[section]
\newtheorem{lemma}[theorem]{Lemma}
\newtheorem{proposition}[theorem]{Proposition}
\newtheorem{corollary}[theorem]{Corollary}

\theoremstyle{definition}
\newtheorem{definition}[theorem]{Definition}
\newtheorem{example}[theorem]{Example}

\theoremstyle{remark}
\newtheorem{remark}[theorem]{Remark}

\numberwithin{equation}{section}

\newcommand{\Des}{\ensuremath{\mathrm{Des}}}

\newcommand{\SPCT}{\ensuremath{\mathsf{SPCT}}}
\newcommand{\SYCT}{\ensuremath{\mathsf{SYCT}}}
\newcommand{\SShT}{\ensuremath{\mathsf{SShT}}}
\newcommand{\StdFill}{\ensuremath{\mathsf{StdFill}}}
\newcommand{\StdTab}{\ensuremath{\mathsf{StdTab}}}
\newcommand{\ShD}{\ensuremath{\mathsf{ShD}}}

\newcommand{\SYRT}{\ensuremath{\mathsf{SYRT}}}

\newcommand{\SET}{\ensuremath{\mathsf{SET}}}

\newcommand{\SIT}{\ensuremath{\mathsf{SIT}}}
\newcommand{\SRCT}{\ensuremath{\mathsf{SRCT}}}
\newcommand{\SRIT}{\ensuremath{\mathsf{SRIT}}}
\newcommand{\SRET}{\ensuremath{\mathsf{SRET}}}
\newcommand{\Rib}{\ensuremath{\mathsf{Rib}}}
\newcommand{\SPYCT}{\ensuremath{\mathsf{SPYCT}}}

\newcommand{\Sym}{\ensuremath{\mathsf{Sym}}}
\newcommand{\QSym}{\ensuremath{\mathsf{QSym}}}

\newcommand{\rect}{\ensuremath{\mathtt{rect}}}
\newcommand{\rw}{\ensuremath{\mathrm{rw}}}
\newcommand{\Peak}{\ensuremath{\mathrm{Peak}}}
\newcommand{\PQSym}{\ensuremath{\mathsf{Peak}}}

\newcommand{\yqs}{\ensuremath{\mathcal{S}}}

\newcommand{\comp}{\ensuremath{\mathrm{comp}}}

\newcommand{\C}{\ensuremath{\mathbb{C}}}

\newcommand{\excise}[1]{}

\newlength\cellsize 

\savebox2{%
\begin{picture}(12,12)
\put(0,0){\line(1,0){12}}
\put(0,0){\line(0,1){12}}
\put(12,0){\line(0,1){12}}
\put(0,12){\line(1,0){12}}
\end{picture}}

\newcommand\cellify[1]{\def\thearg{#1}\def\nothing{}%
\ifx\thearg\nothing\vrule width0pt height\cellsize depth0pt%
  \else\hbox to 0pt{\usebox2\hss}\fi%
  \vbox to 12\unitlength{\vss\hbox to 12\unitlength{\hss$#1$\hss}\vss}}

\newcommand\tableau[1]{\vtop{\let\\=\cr
\setlength\baselineskip{-12000pt}
\setlength\lineskiplimit{12000pt}
\setlength\lineskip{0pt}
\halign{&\cellify{##}\cr#1\crcr}}}

\usepackage[colorinlistoftodos]{todonotes}

\ytableausetup{centertableaux}

\begin{document}


\title{Diagram supermodules for $0$-Hecke-Clifford algebras}  

\author[D. Searles]{Dominic Searles}
\address{Department of Mathematics and Statistics, University of Otago, 730 Cumberland St., Dunedin 9016, New Zealand}
\email{dominic.searles@otago.ac.nz}



\subjclass[2020]{Primary 05E10, 20C08, Secondary 05E05}

\date{February 24, 2022}


\keywords{$0$-Hecke algebra, $0$-Hecke-Clifford algebra, quasisymmetric Schur functions, Schur $Q$-functions, quasisymmetric Schur $Q$-functions}

\begin{abstract}
We introduce a general method for constructing modules for $0$-Hecke algebras and supermodules for $0$-Hecke-Clifford algebras from diagrams of boxes in the plane, and give formulas for the images of these modules in the algebras of quasisymmetric functions and peak functions under the relevant characteristic map. As initial applications, we resolve a question of Jing and Li (2015), introduce a new basis of the peak algebra analogous to the quasisymmetric Schur functions, uncover a new connection between Schur $Q$-functions and quasisymmetric Schur functions, give a representation-theoretic interpretation of families of tableaux used in constructing certain functions in the peak algebra, and establish a common framework for known $0$-Hecke module interpretations of bases of quasisymmetric functions.
\end{abstract}

\maketitle

%
\section{Introduction}
%
\label{sec:introduction}

The \emph{Hecke algebra} $H_n(q)$ is a deformation of the group algebra of the symmetric group $\mathfrak{S}_n$. There is a degenerate case called the \emph{$0$-Hecke algebra} obtained by specializing $q$ to $0$, whose representation theory differs markedly from the case for generic $q$. Representation theory of $0$-Hecke algebras is governed by the Hopf algebra $\QSym$ of \emph{quasisymmetric functions}, which contains the Hopf algebra $\Sym$ of symmetric functions as a subalgebra. Duchamp, Krob, Leclerc and Thibon \cite{DKLT} established an isomorphism, known as \emph{quasisymmetric characteristic}, between the Grothendieck group of $0$-Hecke modules and $\QSym$, under which the images of the simple $0$-Hecke modules are precisely the \emph{fundamental quasisymmetric functions}, a distinguished basis of $\QSym$ originally introduced by Gessel in the study of $P$-partitions \cite{Gessel}. 
Due to the similar role the famous Schur functions play as images of irreducible characters of symmetric groups under the characteristic map, the fundamental quasisymmetric functions can be regarded as a quasisymmetric analogue of the Schur basis of $\Sym$. Representation theory of $0$-Hecke algebras and its connection to $\QSym$ has since been widely studied and developed, e.g., \cite{Duchamp.Hivert.Thibon,  Hivert, Huang, Krob.Thibon}.

There is a similar story regarding the \emph{Hecke-Clifford (super)algebra} $HCl_n(q)$ introduced by Olshanski \cite{Olshanski}. This algebra is a deformation of the Sergeev algebra \cite{Sergeev} constructed from the group algebra of $\mathfrak{S}_n$ and the Clifford algebra. As for the Hecke algebra, there is a degenerate case obtained by specializing $q$ to $0$ called the \emph{$0$-Hecke-Clifford algebra} $HCl_n(0)$, whose representation theory is distinct from the case for generic $q$. Representation theory of $0$-Hecke-Clifford algebras, though less widely studied than that of $0$-Hecke algebras, is known to be governed by the \emph{peak algebra} $\PQSym$: an important Hopf subalgebra of $\QSym$ with applications to, for example, Kazhdan-Lusztig polynomials \cite{Brenti.Caselli}, Eulerian posets \cite{BHvW}, and theory of combinatorial Hopf algebras \cite{Aguiar.Bergeron.Sottile}. The peak algebra has a distinguished basis known as the \emph{peak functions}, introduced by Stembridge in the study of enriched $P$-partitions \cite{Stembridge:enriched}, which comprise a special case (\cite{BMSvW}) of the \emph{shifted quasisymmetric functions} introduced by Billey and Haiman in the context of Schubert calculus \cite{Billey.Haiman}. Bergeron, Hivert and Thibon \cite{BHT} constructed an isomorphism, which we refer to as the \emph{peak characteristic}, between the Grothendieck group of $0$-Hecke-Clifford supermodules and the peak algebra. They moreover determined the simple $0$-Hecke-Clifford supermodules, and showed that their images under the peak characteristic map are (certain constant multiples of) the peak functions. Further structural details regarding the representation theory of $0$-Hecke-Clifford algebras have been obtained by Li \cite{Li}.

There has been much recent activity concerning construction of $0$-Hecke modules whose quasisymmetric characteristics are elements of noteworthy bases of $\QSym$, e.g., \cite{Bardwell.Searles, BBSSZ, NSvWVW:0Hecke, Searles:0Hecke, TvW:1, TvW:2}, and also in determining the structure of these modules, e.g., \cite{CKNO:2, CKNO:1, CKNO:3, JKLO, Koenig}. It is therefore natural to ask the same question for $0$-Hecke-Clifford algebras: can one construct $0$-Hecke-Clifford supermodules whose peak characteristics are noteworthy families or bases of functions in $\PQSym$? Indeed, Jing and Li \cite{Jing.Li} defined the \emph{quasisymmetric Schur $Q$-functions} as a peak algebra analogue of the dual immaculate basis of $\QSym$, and asked whether such a representation-theoretic interpretation of these functions might exist. On the other hand, there are fewer known bases of $\PQSym$ than of $\QSym$, and several bases of $\QSym$ do not have known analogues in $\PQSym$. This motivates a related, more expansive question: can one provide a general method for constructing $0$-Hecke-Clifford supermodules, thus obtaining families of representation-theoretically defined functions in $\PQSym$, and additionally apply this towards finding peak algebra analogues of known bases of $\QSym$?

In this paper, we provide a new general approach that constructs both $0$-Hecke modules and $0$-Hecke-Clifford supermodules from diagrams of boxes in the plane. Given a diagram $D$ with $n$ boxes and a total order on the boxes, we may choose any subset $\StdTab(D)$ (called \emph{standard tableaux}) of the $n!$ different bijective fillings of the boxes of $D$ with the numbers $1, \ldots , n$. If $\StdTab(D)$ satisfies a condition we call \emph{ascent-compatibility} (Definition~\ref{def:ascentcompatible}) with respect to the chosen order, then we obtain both a $0$-Hecke module and $0$-Hecke-Clifford supermodule from $\StdTab(D)$. We call the modules obtained this way \emph{diagram modules}. We establish formulas for the quasisymmetric characteristic of the diagram $0$-Hecke module and, respectively, the peak characteristic of the diagram $0$-Hecke-Clifford supermodule in terms of $\StdTab(D)$, expressed as a nonnegative integral linear combination of fundamental quasisymmetric functions (respectively, peak functions).

A central motivation for our approach in terms of diagrams and tableaux stems from the fact that important families of functions in $\Sym$, $\QSym$ and $\PQSym$ are routinely defined in terms of their expansion into the fundamental quasisymmetric or peak bases, indexed by various families of standard tableaux. A significant advantage of our general construction is that finding $0$-Hecke modules or $0$-Hecke-Clifford supermodules whose characteristics are such a family of functions becomes merely a matter of confirming the ascent-compatibility condition on the associated family of tableaux. Moreover, since our construction method is defined on arbitrary diagrams in the plane, it can be applied directly to all the varied shapes of tableaux that routinely appear in formulas for functions in $\Sym$, $\QSym$ and $\PQSym$, e.g., tableaux of partition shape, composition shape, skew shape, shifted shape, ribbon shape, etc. Notably, the ascent-compatibility condition seems very natural and mild: it is satisfied by the tableaux for every family of functions that we are aware of on which $0$-Hecke modules or $0$-Hecke-Clifford modules have been constructed. Therefore, we expect diagram modules to provide a convenient and useful paradigm for constructing, in a uniform manner, further $0$-Hecke modules and $0$-Hecke-Clifford supermodules whose characteristics are noteworthy families of functions in $\Sym$, $\QSym$ or $\PQSym$.

We then apply this diagram modules framework in several directions. First, we resolve the question of Jing and Li \cite{Jing.Li} in the affirmative. Using a tableau formula for the expansion of  quasisymmetric Schur $Q$-functions in the peak basis \cite{Oguz}, a short verification of ascent-compatibility for the associated family of tableaux immediately yields $0$-Hecke-Clifford supermodules whose peak characteristics are precisely the quasisymmetric Schur $Q$-functions. We also analyze these supermodules further, demonstrating that they are always cyclic and explicitly determining a generator.

Next, we consider the question of obtaining new families of functions in $\PQSym$ as peak characteristics of $0$-Hecke-Clifford supermodules, and in particular, obtaining new bases of $\PQSym$ as analogues of important bases of $\QSym$. Constructing diagram $0$-Hecke modules whose quasisymmetric characteristics are a particular family of functions in $\QSym$ immediately yields corresponding diagram $0$-Hecke-Clifford supermodules and formulas for their peak characteristics, thus a family of representation-theoretically defined functions in $\PQSym$. We apply our framework to the well-studied (Young) quasisymmetric Schur basis of $\QSym$ \cite{HLMvW11:QS, LMvWbook}, and thus obtain a new basis of $\PQSym$ as an analogue of the Young quasisymmetric Schur functions, which we call the \emph{peak Young quasisymmetric Schur functions}.

We apply our techniques further to analyze properties of this new basis. The \emph{Schur $Q$-functions}, a family of symmetric functions integral in the study of projective representations of symmetric groups \cite{Stembridge:shifted}, are the peak algebra analogue of the Schur functions. Similarly to how the fundamental quasisymmetric functions positively refine the Schur functions, the peak functions positively refine the Schur $Q$-functions. Moreover, the Young quasisymmetric Schur functions provide an especially elegant refinement of the Schur functions \cite{LMvWbook}, and so it is natural to ask how our new basis relates to the Schur $Q$-functions. To accomplish this, we apply our construction method once more to the standard shifted tableaux defining Schur $Q$-functions. We thus obtain $0$-Hecke-Clifford supermodules whose peak characteristics are the Schur $Q$-functions, and we prove each of these supermodules is isomorphic to a supermodule for some peak Young quasisymmetric Schur function.  As a result, the peak Young quasisymmetric Schur functions in fact contain the Schur $Q$-functions as a subset. Therefore, in addition to providing a peak algebra analogue of Young quasisymmetric Schur functions, the peak Young quasisymmetric Schur functions naturally extend the Schur $Q$-functions to a basis of the peak algebra.

We also consider the $0$-Hecke modules that arise, as a consequence of the diagram modules framework, alongside the $0$-Hecke-Clifford supermodules for the Schur $Q$-functions. This leads to a further connection between Schur $Q$-functions and Young quasisymmetric Schur functions: letting the aforementioned isomorphism descend to the case of $0$-Hecke modules, we find that the quasisymmetric characteristics of the $0$-Hecke modules arising from the standard shifted tableaux generating Schur $Q$-functions are, in fact, Young quasisymmetric Schur functions.

We expect that similar applications of the diagram modules framework to other important families of functions in $\QSym$ will yield new families of functions in $\PQSym$ with similarly notable properties and relationships to other families.

A further consequence of the diagram modules framework is a representation-theoretic interpretation of the \emph{marked standard (shifted) tableaux} used in the study of Schur $Q$-functions \cite{Assaf:shifted} and \emph{marked standard peak composition tableaux} used in the study of quasisymmetric Schur $Q$-functions \cite{Oguz}. Our construction method realizes these families of tableaux as precisely the basis elements of the $0$-Hecke-Clifford supermodules whose peak characteristics are, respectively, the Schur $Q$-functions and the quasisymmetric Schur $Q$-functions. In general, for any ascent-compatible family of standard tableaux having any underlying diagram shape, the corresponding family of marked standard tableaux (in which the integers appearing in the standard tableaux can be either marked or unmarked) forms the basis of a diagram $0$-Hecke-Clifford supermodule.

Demonstrating the wide applicability of the diagram modules framework, we show that the $0$-Hecke modules constructed for notable bases of $\QSym$ in each of \cite{Bardwell.Searles, BBSSZ, NSvWVW:0Hecke, Searles:0Hecke, TvW:1} are cases of diagram $0$-Hecke modules. Extending this further, we show that the $0$-Hecke modules constructed from permuted tableaux in \cite{CKNO:2, JKLO, TvW:2}, and the projective indecomposable $0$-Hecke modules constructed in terms of ribbon tableaux in \cite{Huang}, are likewise obtained as diagram $0$-Hecke modules. Diagram $0$-Hecke modules moreover generalize the family of \emph{weak Bruhat interval modules} considered by Jung, Kim, Lee and Oh \cite{JKLO} as an alternative framework for studying the $0$-Hecke modules in \cite{Bardwell.Searles, BBSSZ, Searles:0Hecke, TvW:1, TvW:2}. We prove that every weak Bruhat interval module can be realized as a diagram $0$-Hecke module; on the other hand, there exist diagram $0$-Hecke modules that are not isomorphic to any weak Bruhat interval module. Diagram modules thus provide a universal framework not only for known $0$-Hecke modules arising from notable bases of $\QSym$, but many further notable families of $0$-Hecke modules in addition.

We briefly conclude with a sample of avenues for further research arising from diagram modules. There are two overarching directions to follow: firstly in better understanding the structure of the diagram modules themselves, and secondly in applying the diagram modules framework to obtain new families of functions in $\PQSym$, and to determine their properties and connections to known structures.

%
\section{Background}
%
\label{sec:background}

\subsection{Compositions, quasisymmetric functions and peak functions}

A \emph{composition} $\alpha = (\alpha_1, \ldots , \alpha_k)$ is a finite sequence of positive integers. The entries $\alpha_i$ of a composition $\alpha$ are called the \emph{parts} of $\alpha$. The number of parts of $\alpha$ is called the \emph{length} of $\alpha$, denoted $\ell(\alpha)$. 
If the parts of $\alpha$ sum to $n$, $\alpha$ is called a \emph{composition of $n$} and we write $\alpha \vDash n$.

Let $[n]$ denote the set $\{1, \ldots , n\}$. The \emph{descent set} $\Des(\alpha)$ of a composition $\alpha= (\alpha_1, \ldots , \alpha_k)\vDash n$ is the subset $\{\alpha_1, \alpha_1+\alpha_2, \ldots ,  \alpha_1+\alpha_2 + \cdots \alpha_{k-1}\}$ of $[n-1]$. The map $\alpha \mapsto \Des(\alpha)$ is a bijection between compositions of $n$ and subsets of $[n-1]$. Given a subset $X$ of $[n-1]$, we write $\comp_n(X)$ for the composition of $n$ whose descent set is $X$. 
For example, $\Des(2,3,1,1) = \{2,5,6\}\subseteq [7]$ and $\comp_8(\{1,2,5\}) = (1,1,3,3)\vDash 8$. 

Let $\C[[x_1, x_2, \ldots ]]$ denote the Hopf algebra of formal power series of bounded degree in infinitely many commuting variables. The Hopf algebra $\QSym$ of quasisymmetric functions is a subalgebra of $\C[[x_1, x_2, \ldots ]]$, and the Hopf algebra $\Sym$ of symmetric functions is a subalgebra of $\QSym$. Bases of $\QSym$ are indexed by compositions. The \emph{fundamental quasisymmetric functions} \cite{Gessel} form an important basis of $\QSym$. Given $\alpha\vDash n$, the fundamental quasisymmetric function $F_\alpha$ is defined by  
\[F_\alpha = \sum x_{i_1}x_{i_2} \cdots x_{i_n},\]
where the sum is over sequences $i_1, \ldots , i_n$ of positive integers such that $i_1 \le \cdots \le i_n$ and $i_j < i_{j+1}$ whenever $j\in \Des(\alpha)$.

\begin{example}
Let $\alpha = (1,2,1)\vDash 4$. Then $\Des(\alpha) = \{1,3\}$, and we have 
\[F_{(1,2,1)} = \sum_{i<j\le k<\ell} x_ix_jx_kx_\ell = x_1x_2^2x_3+x_1x_2^2x_4+ x_2x_3^2x_4 + \cdots + x_1x_2x_3x_4 + x_1x_2x_3x_5+ \cdots \]
\end{example}

A \emph{peak composition} is a composition whose parts (except possibly the last part) are all greater than $1$. Therefore, under the map taking a composition to its descent set, peak compositions of $n$ are in bijection the with subsets of $[n-1]$ that do not contain $1$ and do not contain any pair of consecutive integers. 
For an arbitrary set $X$ of positive integers, the \emph{peak set of $X$}, denoted $\Peak(X)$, is the set $\{i\in X : i > 1 \mbox{ and } i-1 \notin X\}$. For an arbitrary composition $\alpha$, the \emph{peak set of $\alpha$} is the set $\Peak(\Des(\alpha))$, which we write as $\Peak(\alpha)$ for short. For example, if $\alpha = (1,2,2,1,3)$, then $\Des(\alpha) = \{1,3,5,6\}$ and $\Peak(\alpha) = \Peak(\{1,3,5,6\}) = \{3,5\}$. Note that $\Peak(\alpha) = \Des(\alpha)$ if and only if $\alpha$ is a peak composition.

The \emph{peak algebra} $\PQSym$ is a Hopf subalgebra of $\QSym$, and bases of $\PQSym$ are indexed by peak compositions. An important basis of $\PQSym$ is given by the \emph{peak functions} \cite{Stembridge:enriched}. Given a peak composition $\alpha$, the peak function $K_\alpha$ is defined by 
\begin{equation}\label{eqn:peakintofundamental}
K_\alpha = 2^{|\Peak(\alpha)|+1}\sum_{\beta \, : \, \Peak(\alpha)\subseteq \Des(\beta)\triangle(\Des(\beta)+1)} F_\beta,
\end{equation}
where $X\triangle Y$ is the symmetric difference of sets $X$ and $Y$, and $X+1 = \{x+1: x\in X\}$.

\subsection{0-Hecke algebras and 0-Hecke-Clifford algebras}
The \emph{Hecke algebra} $H_n(q)$ is the $\C$-algebra with generators $\pi_1, \ldots , \pi_{n-1}$, subject to the relations
\begin{align}\label{eqn:qHecke}
\pi_i^2 & =  (q-1)\pi_i+q & & \mbox{ for all } 1 \le i \le n-1 \nonumber \\
\pi_i\pi_j & =  \pi_j\pi_i  & &\mbox{ for all } i, j  \mbox{ such that } |i-j|\ge 2 \\
\pi_i\pi_{i+1}\pi_i & = \pi_{i+1}\pi_i\pi_{i+1} & &   \mbox{ for all } 1\le i \le n-2. \nonumber
\end{align}
We denote these operators by $\pi_i$ rather than the more usual $T_i$ in order to reserve the letter $T$ for tableaux that appear in later sections. 
Notice that specializing $q$ to $1$ yields the group algebra of the symmetric group $\mathfrak{S}_n$. 

The \emph{$0$-Hecke algebra $H_n(0)$} is obtained by specializing $q$ to $0$, so that the first relation in (\ref{eqn:qHecke}) becomes $\pi_i^2=-\pi_i$. 

There are $2^{n-1}$ simple $H_n(0)$-modules, all of which are one-dimensional \cite{Norton}. These simple modules may be indexed by the $2^{n-1}$ compositions of $n$; let ${\bf F}_\alpha$ denote the simple module corresponding to the composition $\alpha$ and let $\{v_\alpha\}$ be a basis of ${\bf F}_\alpha$. The structure of ${\bf F}_\alpha$ as an $H_n(0)$-module is given by 
\begin{align}\label{eqn:irreps}
\pi_iv_\alpha = \begin{cases} -v_\alpha & \mbox{ if } i\in \Des(\alpha) \\
					       0 & \mbox{ if } i \notin \Des(\alpha).
					       \end{cases}
 \end{align}
 
 Let $G_0(H_n(0))$ denote the Grothendieck group of the category of finite-dimensional $H_n(0)$-modules, and let $\mathcal{G} = \bigoplus_{n\ge 0} \mathcal{G}_0(H_n(0))$. For an $H_n(0)$-module ${\bf N}$, let $[{\bf N}]$ denote its isomorphism class. Duchamp, Krob, Leclerc and Thibon \cite{DKLT} defined an algebra isomorphism $ch:\mathcal{G}\rightarrow \QSym$, under which
\begin{equation}\label{eqn:qsymchar}
ch([{\bf F}_\alpha]) = F_\alpha.
\end{equation}
The \emph{quasisymmetric characteristic} of an $H_n(0)$-module ${\bf N}$ is the quasisymmetric function $ch([{\bf N}])$.

The \emph{$0$-Hecke-Clifford algebra} $HCl_n(0)$ is the algebra generated by $\pi_1, \ldots , \pi_{n-1}$ and $c_1, \ldots , c_n$, where the $\pi_i$ generate the $0$-Hecke algebra (i.e., satisfy the relations (\ref{eqn:qHecke}) at $q=0$), and the $c_j$ generate the \emph{Clifford algebra} $Cl_n$, i.e., satisfy the relations 
\begin{align}
c_i^2 & =  -1 && \mbox{ for all } 1 \le i \le n \\ \nonumber
c_ic_j & =  -c_jc_i  && \mbox{ for all } i, j  \mbox{ such that } i\neq j,
\end{align}

and the $\pi_i$ and $c_j$ additionally satisfy the cross-relations 
\begin{align}\label{eqn:crossrelations}
\pi_ic_j & =  c_j\pi_i & & \mbox{ for } j\neq i, i+1 \nonumber \\ 
\pi_ic_i & =  c_{i+1}\pi_i  & &\mbox{ for all } 1\le i \le n-1 \\
(\pi_i+1)c_{i+1} & = c_i(\pi_i+1) & &   \mbox{ for all } 1\le i \le n-1. \nonumber
\end{align}
The $0$-Hecke-Clifford algebra is the $q=0$ specialization of the Hecke-Clifford algebra $HCl_n(q)$ introduced by Olshanksi \cite{Olshanski}. 
The $0$-Hecke-Clifford algebra is a superalgebra: it is graded by $\mathbb{Z}_2$ with each $\pi_i$ having degree $0$ and each $c_j$ having degree~$1$.

Bergeron, Hivert and Thibon \cite{BHT} defined $HCl_n(0)$-supermodules $\mathbf{M}_\alpha$ as the modules induced from the simple $H_n(0)$-modules ${\bf F}_\alpha$. Letting $\mathcal{\widetilde{G}} = \bigoplus_{n\ge 0} \mathcal{G}_0(HCl_n(0))$, they defined an algebra isomorphism $\widetilde{ch}:\mathcal{\widetilde{G}}\rightarrow \Peak$, which we refer to as the \emph{peak characteristic}. Under this isomorphism, one has 
\begin{equation}\label{eqn:peakcharM}
\widetilde{ch}([\mathbf{M}_\alpha]) = K_{\comp_n(\Peak(\alpha))}.
\end{equation}

For $X = \{i_1 < i_2 < \cdots < i_k\} \subseteq [n]$, define $c_X = c_{i_1}c_{i_2}\cdots c_{i_k}$. Each $\mathbf{M}_\alpha$ has basis $\{c_Xv_\alpha : X\subseteq [n]\}$, where $v_\alpha$ is the single basis element of ${\bf F}_\alpha$. 
The image of each basis element $c_Xv_\alpha$ under $\pi_i$ can be determined from (\ref{eqn:irreps}) and the cross-relations (\ref{eqn:crossrelations}). We reproduce these formulas from \cite{BHT} as (\ref{eqn:MIdescent}) and (\ref{eqn:MIascent}) below, since we will need them in Section~\ref{sec:modules}. 

If $i\in \Des(\alpha)$, then
\begin{equation}\label{eqn:MIdescent}
\pi_i(c_Xv_\alpha) = \begin{cases} -c_Xv_\alpha & \mbox{ if } i, i+1 \notin X \\
                                                        -c_{(X\setminus \{i\})\cup \{i+1\}}v_\alpha & \mbox{ if } i \in X \mbox{ and } i+1 \notin X \\
                                                        -c_Xv_\alpha & \mbox{ if } i \notin X \mbox{ and } i+1 \in X \\
                                                        c_{X\setminus \{i,i+1\}}v_\alpha & \mbox{ if } i, i+1 \in X.
                                                        \end{cases}
\end{equation}

If $i\notin \Des(\alpha)$, then 
\begin{equation}\label{eqn:MIascent}
\pi_i(c_Xv_\alpha) = \begin{cases} 0 & \mbox{ if } i, i+1 \notin X \\
                                                        0 & \mbox{ if } i \in X \mbox{ and } i+1 \notin X \\
                                                        -c_Xv_\alpha + c_{(X\setminus \{i+1\})\cup \{i\}}v_\alpha & \mbox{ if } i \notin X \mbox{ and } i+1 \in X \\
                                                        -c_Xv_\alpha + c_{X\setminus \{i,i+1\}}v_\alpha & \mbox{ if } i, i+1 \in X.
                                                        \end{cases}
\end{equation}

%
\section{$0$-Hecke-Clifford supermodules from diagrams}\label{sec:modules}
%
In this section we give a construction method that produces $0$-Hecke modules and $0$-Hecke-Clifford supermodules from diagrams in the plane, and we compute their quasisymmetric (respectively, peak) characteristics.

\subsection{Ascent-compatibility and construction of the modules}
Let $D$ denote a diagram consisting of $n$ boxes in the coordinate plane. Define a \emph{standard filling} of $D$ to be a bijective assignment of the integers $1, \ldots , n$ to the boxes of $D$, and let $\StdFill(D)$ denote the set of all $n!$ standard fillings of $D$. If $T\in \StdFill(D)$, we say that the box to which $i$ is assigned has \emph{entry} $i$, and we depict $T$ by writing the entries inside their boxes in $D$. Fix a total ordering on the boxes of $D$: we call this the \emph{reading order} on $D$. If $T\in \StdFill(D)$, the \emph{reading word} $\rw(T)$ of $T$ is the permutation of $[n]$ obtained by writing the entries of $T$ in reading order from left to right. 

Let $\StdTab(D)$ be an arbitrary nonempty subset of $\StdFill(D)$. Define a function $\Des: \StdTab(D) \rightarrow [n-1]$ by 
\begin{equation}\label{eqn:rwdescent}
\Des(T) = \{i \in T : i \mbox{ is to the right of } i+1 \mbox{ in } \rw(T)\}.
\end{equation}
We call $\Des(T)$ the \emph{descent set} of $T$ and call an element of $\Des(T)$ a \emph{descent} of $T$. An element of $[n-1]$ not in $\Des(T)$ is called an \emph{ascent} of $T$. 
For any $T\in \StdFill(D)$ and any $1\le i \le n-1$, let $s_iT$ denote the element of $\StdFill(D)$ obtained by exchanging the entries $i$ and $i+1$ in $T$. The ascents of $T$ are partitioned into \emph{attacking} and \emph{nonattacking} ascents: an ascent $i$ is attacking if $s_iT\notin \StdTab(D)$, and nonattacking if $s_iT\in \StdTab(D)$.

\begin{example}\label{ex:rw}
Let $n=3$ and let $D =  \tableau{ & & {\ } \\ {\ } & {\ } }\,$, with reading order defined by reading the higher box first and then the lower boxes from left to right. Let 
\[\StdTab(D) = \left\{ R = \tableau{ & & 3 \\ 1 & 2 } \, , \,\,\, S = \tableau{ & & 1 \\ 2 & 3 } \, , \,\,\, T = \tableau{ & & 1 \\ 3 & 2 } \right\}.\]  
Here $\rw(R) = 312$, $\rw(S) = 123$, $\rw(T)=132$. We have $\Des(R) = \{2\}$, $\Des(S) = \emptyset$, $\Des(T)=\{2\}$. In $S$, the ascent $1$ is attacking since $s_1S \notin \StdTab(D)$, while the ascent $2$ is nonattacking since $s_2S = T \in \StdTab(D)$.
\end{example}

To facilitate our main definition below, we also say that $T$ has an \emph{ascent in positions $(r,s)$} if $r<s$ and $\rw(T)_{s}-\rw(T)_{r}=1$, where $\rw(T)_r$ denotes the $r$th entry (from the left) of $\rw(T)$. For example, the standard tableau $R$ from Example~\ref{ex:rw} has $\rw(R)=312$, and thus the ascent $1$ in $R$ is an ascent in positions $(2,3)$.

\begin{definition}\label{def:ascentcompatible}
The set $\StdTab(D)$ is said to be \emph{ascent-compatible} if whenever $T, T' \in \StdTab(D)$ both have an ascent in positions $(r,s)$, this ascent is attacking in $T$ if and only if it is attacking in $T'$. 
\end{definition}
 In other words, for ascents involving the same pair of boxes of $D$, whether the ascent is attacking or nonattacking does not depend on which $T\in \StdTab(D)$ the ascent appears in.

\begin{example}
The family $\StdTab(D)$ of standard tableaux defined in Example~\ref{ex:rw} is not ascent-compatible. We have $\rw(R) = 312$ and $\rw(S) = 123$, so $1$ is an ascent of $R$ in positions $(2,3)$ and $2$ is an ascent of $S$ in positions $(2,3)$. Since these ascents are in the same positions, ascent-compatibility requires they both be attacking or both be nonattacking. But $s_1R\notin\StdTab(D)$ while $s_2S = T\in \StdTab(D)$, so the ascent of $R$ in positions $(2,3)$ is attacking but the ascent of $S$ in positions $(2,3)$ is nonattacking. 
\end{example}

We now construct $H_n(0)$-modules from ascent-compatible sets $\StdTab(D)$. Let ${\bf N}_{\StdTab(D)}$ denote the complex span of $\StdTab(D)$. Define operators $\pi_1, \ldots , \pi_{n-1}$ on ${\bf N}_{\StdTab(D)}$ by
\begin{equation}\label{eqn:0Hecke}
\pi_iT = \begin{cases} -T & \mbox{ if } i\in \Des(T)  \\
                                        0 & \mbox{ if } i \notin \Des(T), i \mbox{ is attacking } \\
                                        s_iT & \mbox{ if } i \notin \Des(T), i \mbox{ is nonattacking }
                                        \end{cases}
\end{equation}
for each $T\in \StdTab(D)$.

\begin{example}\label{ex:0H}
Let $n=3$ and $D = \tableau{  {\ } & {\ } \\ & {\ } }\, $, with reading order defined by reading the lower box first and then the higher boxes from left to right. Let
\[\StdTab(D) = \left\{ R = \tableau{ 1 & 3 \\ & 2 } \, , \,\,\, S = \tableau{ 2 & 3 \\ & 1 } \, ,\,\,\, T = \tableau{ 3 & 2 \\ & 1 } \right\}.\]
Then $\rw(R) = 213$, $\rw(S) = 123$ and $\rw(T)=132$. Only $S$ has an ascent in positions $(1,2)$ or $(2,3)$; $R$ and $T$ both have ascents in positions $(1,3)$ but both these ascents are attacking, hence $\StdTab(D)$ is ascent-compatible. By (\ref{eqn:0Hecke}), we have
\[\pi_1R = -R, \,\,\,\,\,\, \pi_2R = 0,  \,\,\,\,\,\, \pi_1S = R, \,\,\,\,\,\, \pi_2S=T,  \,\,\,\,\,\, \pi_1T=0,  \,\,\,\,\,\, \pi_2T=-T.\]
\end{example}

\begin{theorem}\label{thm:0Hecke}
Let $D$ be a diagram of $n$ boxes in the plane and $\StdTab(D)$ an ascent-compatible subset of $\StdFill(D)$. Then the operators $\{\pi_i : 1\le i \le n-1\}$ define a $0$-Hecke action on ${\bf N}_{\StdTab(D)}$.
\end{theorem}
\begin{proof}
Let $T\in \StdTab(D)$. By definition, if $i$ is a nonattacking ascent in $T$, then $s_iT\in \StdTab(D)$. Hence $\pi_iT\in {\bf N}_{\StdTab(D)}$ for all $i$.

We now prove that the relations (\ref{eqn:qHecke}) are satisfied at $q=0$. First we show that $\pi_i^2T=-\pi_iT$. If $\pi_iT = -T$, then $\pi_i^2T = \pi_i(-T) = -\pi_iT$. If $\pi_iT=0$, then $\pi_i^2T = 0 = -\pi_iT$. If $\pi_iT = s_iT$, then $\pi_i^2T = \pi_is_iT = -s_iT = -\pi_iT$, where the second equality follows since if $i$ is an ascent in $T$ then $i$ is a descent in $s_iT$. Hence $\pi_i^2=\pi_i$.

Next we show that if $|i-j|\ge 2$, then $\pi_i\pi_jT = \pi_j\pi_iT$. First suppose $\pi_iT = 0$. Then $\pi_j\pi_iT=0$. If $\pi_jT = 0$ or $\pi_jT=-T$ then $\pi_i\pi_jT=0$. If $\pi_jT =s_jT$ then $\pi_i\pi_jT = \pi_is_jT$. Since $i$ and $i+1$ occupy the same boxes in $s_jT$ as they do in $T$, and this ascent is attacking in $T$, we have $\pi_is_jT=0$ by ascent-compatibility.

Now suppose $\pi_iT=-T$. If $\pi_jT=0$ then $\pi_i\pi_jT = 0 = \pi_j\pi_iT$. If $\pi_jT=-T$ then $\pi_i\pi_jT = T = \pi_j\pi_iT$. If  $\pi_jT =s_jT$ then $\pi_j\pi_iT = -s_jT$, while $\pi_i\pi_jT = \pi_is_jT$. But $\pi_is_jT=-s_jT$ since $i$ is a descent in $T$ and therefore a descent in $s_jT$.

Now suppose $\pi_iT=s_iT$. Then $\pi_j\pi_iT = \pi_js_iT$. If $\pi_jT=0$ then $\pi_i\pi_jT=0$, and $\pi_js_iT$ must be $0$ by ascent-compatibility since $j$ and $j+1$ occupy the same boxes in $s_iT$ as they do in $T$. If $\pi_jT=-T$ then $\pi_i\pi_jT = -s_iT$, and $\pi_js_iT$ must be equal to $-s_iT$ since $j$ is a descent in $T$ and therefore a descent in $s_iT$. If $\pi_jT=s_jT$ then $\pi_i\pi_jT = \pi_is_jT$ which is equal to $s_is_jT$ by ascent-compatibility since $i$ and $i+1$ occupy the same boxes in $s_jT$ as they do in $T$. Similarly $\pi_j\pi_iT = \pi_js_iT$ which is equal to $s_js_iT$ by ascent-compatibility since $j$ and $j+1$ occupy the same boxes in $s_iT$ as they do in $T$. Since $|i-j|\ge 2$,  we have $s_is_jT = s_js_iT$. Hence $\pi_i\pi_j = \pi_j\pi_i$ when $|i-j|\ge 2$.

Finally, we establish $\pi_i\pi_{i+1}\pi_iT = \pi_{i+1}\pi_i\pi_{i+1}T$ via the following cases.
\begin{enumerate}
\item $\pi_iT=0$
\item $\pi_iT=-T$
 \begin{enumerate}
 \item $\pi_{i+1}T=0$
 \item $\pi_{i+1}T=-T$
 \item $\pi_{i+1}T=s_{i+1}T$
 \end{enumerate}
\item $\pi_iT=s_iT$
 \begin{enumerate}
 \item $\pi_{i+1}T=0$
 \item $\pi_{i+1}T=-T$
 \item $\pi_{i+1}T=s_{i+1}T$
 \end{enumerate}
\end{enumerate}
\noindent
(1): Here we have $\pi_i\pi_{i+1}\pi_iT=0$. If $\pi_{i+1}T=0$ or $\pi_{i+1}T=-T$ then $\pi_{i+1}\pi_i\pi_{i+1}T = 0$. Now suppose $\pi_{i+1}T = s_{i+1}T$. Then both $i$ and $i+1$ are ascents in $T$, so $i, i+1, i+2$ appear in increasing order in $\rw(T)$ from left to right. Hence $i$ is an ascent in $s_{i+1}T$. If $\pi_i(s_{i+1}T) = 0$ we are done. If $\pi_i(s_{i+1}T) = s_is_{i+1}T$ then by ascent-compatibility, $i+1$ must be an attacking ascent in $s_is_{i+1}T$ since $i$ is an attacking ascent in $T$ and the boxes occupied by $i+1$ and $i+2$ in $s_is_{i+1}T$ are precisely those occupied by $i$ and $i+1$ in $T$. Hence $0 = \pi_{i+1}s_is_{i+1}T = \pi_{i+1}\pi_i\pi_{i+1}T$.

\noindent
(2)(a): Here we have $\pi_{i+1}\pi_i\pi_{i+1}T=0$, and also $\pi_i\pi_{i+1}\pi_iT = -\pi_i\pi_{i+1}T = 0$.

\noindent
(2)(b): Here $\pi_{i+1}\pi_i\pi_{i+1}T= -T = \pi_i\pi_{i+1}\pi_iT$.

\noindent
(2)(c): Here $\pi_i\pi_{i+1}\pi_iT = -\pi_i(s_{i+1}T)$ and $\pi_{i+1}\pi_i\pi_{i+1}T = \pi_{i+1}\pi_is_{i+1}T$. 

If $\pi_i(s_{i+1}T)=0$ then both sides are $0$. If $\pi_i(s_{i+1}T) = -s_{i+1}T$, then $\pi_i\pi_{i+1}\pi_iT=s_{i+1}T$ and $\pi_{i+1}\pi_i\pi_{i+1}T = -\pi_{i+1}s_{i+1}T = s_{i+1}T$ since $i+1$ is a descent in $s_{i+1}T$. Now suppose $\pi_is_{i+1}T = s_is_{i+1}T$. Then $\pi_i\pi_{i+1}\pi_iT = -s_is_{i+1}T$ and $\pi_{i+1}\pi_is_{i+1}T = \pi_{i+1}s_is_{i+1}T$. Hence it suffices to show that $i+1$ is a descent in $s_is_{i+1}T$. We have $i$ is a descent in $T$ and an ascent in $s_{i+1}T$, so $i+1$ appears first in $\rw(T)$, followed by $i$, followed by $i+2$. Therefore $i+1$ is right of $i+2$ in $\rw(s_is_{i+1}T)$, as required.

\noindent
(3)(a): Here we have $\pi_{i+1}\pi_i\pi_{i+1}T=0$ and $\pi_i\pi_{i+1}\pi_iT =  \pi_i\pi_{i+1}s_iT$. Both $i$ and $i+1$ are ascents in $T$, hence $i$ appears first in $\rw(T)$, followed by $i+1$, followed by $i+2$. Therefore $i+1$ is an ascent in $s_iT$. If $\pi_{i+1}s_iT=0$ we are done. If $\pi_{i+1}s_iT=s_{i+1}s_iT$, then by ascent-compatibility $i$ is an attacking ascent in $s_{i+1}s_iT$, since $i+1$ is an attacking ascent in $T$ and the boxes occupied by $i$ and $i+1$ in $s_is_{i+1}T$ are exactly those occupied by $i+1$ and $i+2$ in $T$. Hence $\pi_is_{i+1}s_iT = 0$, and so $\pi_i\pi_{i+1}\pi_iT=0$.

\noindent
(3)(b): Here we have $\pi_{i+1}\pi_i\pi_{i+1}T = -\pi_{i+1}s_iT$ and $\pi_i\pi_{i+1}\pi_iT =\pi_i\pi_{i+1}s_iT$.

If $\pi_{i+1}s_iT=0$, then both sides are $0$. If $\pi_{i+1}s_iT=-s_iT$, then $\pi_{i+1}\pi_i\pi_{i+1}T = s_iT$ and $\pi_i\pi_{i+1}\pi_iT =-\pi_is_iT = s_iT$ since $i$ is a descent in $s_iT$. Now suppose $\pi_{i+1}s_iT=s_{i+1}s_iT$. Then $\pi_{i+1}\pi_i\pi_{i+1}T = -s_{i+1}s_iT$ and $\pi_i\pi_{i+1}\pi_iT =\pi_is_{i+1}s_iT$.  Hence it is enough to show that $i$ is a descent in $s_{i+1}s_iT$. By our assumptions, $i$ is an ascent in $T$, $i+1$ is a descent in $T$, and $i+1$ is an ascent in $s_iT$, which together imply that in $\rw(T)$, $i$ appears first, followed by $i+2$, followed by $i+1$. Therefore $i$ appears right of $i+1$ in $\rw(s_{i+1}s_iT)$, as required.

\noindent
(3)(c): Here we have $\pi_i\pi_{i+1}\pi_iT =\pi_i\pi_{i+1}s_iT$ and $\pi_{i+1}\pi_i\pi_{i+1}T =  \pi_{i+1}\pi_is_{i+1}T$. So $i$ appears first in $\rw(T)$, followed by $i+1$, followed by $i+2$, and the ascents $i$ and $i+1$ are both nonattacking in $T$. 

Moreover, $i+1$ is an ascent in $s_iT$ and $i$ is an ascent in $s_{i+1}T$, and the boxes occupied by $i, i+1$ in $s_{i+1}T$ are exactly the boxes occupied by $i+1, i+2$ in $s_iT$. Hence by ascent-compatibility, either $\pi_{i+1}s_iT = 0 = \pi_is_{i+1}T$ or $\pi_{i+1}s_{i}T=s_{i+1}s_{i}T$ and $\pi_is_{i+1}T=s_is_{i+1}T$.

In the former case, we are done. In the latter case, $i$ and $i+1$ occupy the same boxes in $s_{i+1}s_iT$ as $i+1$ and $i+2$ occupy in $T$. Since this ascent is nonattacking in $T$, by ascent-compatibility we have $\pi_is_{i+1}s_{i}T = s_{i}s_{i+1}s_{i}T$. Similarly, by ascent-compatibility we have $\pi_{i+1}s_is_{i+1}T =  s_{i+1}s_is_{i+1}T$. Therefore $\pi_i\pi_{i+1}\pi_iT = s_{i}s_{i+1}s_{i}T$ and $\pi_{i+1}\pi_i\pi_{i+1}T = s_{i+1}s_is_{i+1}T$. These are the same since $s_{i+1}s_is_{i+1}=s_{i}s_{i+1}s_{i}$.
\end{proof}

For $\StdTab(D)$ ascent-compatible, we now define a $HCl_n(0)$-supermodule $\widetilde{{\bf N}}_{\StdTab(D)}$ as the module induced from the $H_n(0)$-module ${\bf N}_{\StdTab(D)}$, i.e.,
\[\widetilde{{\bf N}}_{\StdTab(D)} = {\rm Ind}_{H_n(0)}^{HCl_n(0)}{\bf N}_{\StdTab(D)}.\] 
A basis of $\widetilde{{\bf N}}_{\StdTab(D)}$ is $\{c_XT : X\subseteq [n], T\in \StdTab(D)\}$.

\subsection{Peak and quasisymmetric characteristics of diagram modules}
For the remainder of this section, we assume $\StdTab(D)$ is ascent-compatible unless otherwise stated. Towards computing the peak characteristic of $\widetilde{{\bf N}}_{\StdTab(D)}$ in terms of $\StdTab(D)$, we first establish precise formulas for the images of the basis elements of $\widetilde{{\bf N}}_{\StdTab(D)}$ under each $\pi_i$.

\begin{proposition}\label{prop:Dcrossrelations}
Let $T\in \StdTab(D)$ and $X\subseteq [n]$. For each $1\le i \le n-1$, the following formulas compute $\pi_i(c_XT)$ in $\widetilde{{\bf N}}_{\StdTab(D)}$.

If $i\in \Des(T)$, we have 
\begin{equation}\label{eqn:descent}\pi_i(c_XT) = 
\begin{cases} -c_XT & \mbox{ if } i, i+1 \notin X \\
		      -c_{(X\setminus \{i\})\cup \{i+1\}}T & \mbox{ if } i\in X, i+1\notin X \\
		      -c_XT & \mbox{ if } i\notin X, i+1\in X \\
		      c_{X\setminus \{i,i+1\}}T & \mbox{ if } i, i+1 \in X.
\end{cases}
\end{equation}

If $i\notin \Des(T)$ is attacking, we have 
\begin{equation}\label{eqn:ascent0}
\pi_i(c_XT) = 
\begin{cases} 0 & \mbox{ if } i, i+1 \notin X \\
		      0 & \mbox{ if } i\in X, i+1\notin X \\
		      -c_XT+c_{(X\setminus \{i+1\})\cup \{i\}}T & \mbox{ if } i\notin X, i+1\in X \\
		      -c_XT+c_{X\setminus \{i,i+1\}}T & \mbox{ if } i, i+1 \in X.
\end{cases}
\end{equation}

If $i\notin \Des(T)$ is nonattacking, we have 
\begin{equation}\label{eqn:ascentswap}
\pi_i(c_XT) = 
\begin{cases} c_Xs_iT & \mbox{ if } i, i+1 \notin X \\
		      c_{(X\setminus \{i\})\cup \{i+1\}}s_iT & \mbox{ if } i\in X, i+1\notin X \\
		      -c_XT +  c_{(X\setminus \{i+1\})\cup \{i\}}T +  c_{(X\setminus \{i+1\})\cup \{i\}}s_iT & \mbox{ if } i\notin X, i+1\in X \\
		      -c_XT + c_{X\setminus \{i,i+1\}}T - c_Xs_iT & \mbox{ if } i, i+1 \in X.
\end{cases}
\end{equation}
\end{proposition}

\begin{proof}
By (\ref{eqn:crossrelations}), $\pi_i$ commutes with $c_j$ whenever $j\neq i, i+1$ and thus we only need to consider whether $i$, $i+1$ are in $X$ or not. 

If $i,i+1 \notin X$ then $\pi_i(c_XT) = c_X\pi_iT$, so the first line of each of (\ref{eqn:descent}), (\ref{eqn:ascent0}) and (\ref{eqn:ascentswap}) follows from (\ref{eqn:0Hecke}).

Let $p_i(X)$ denote the number of elements of $X$ that are strictly smaller than $i$. 
If $i \in X$ and $i+1 \notin X$ then
\begin{align*}
\pi_ic_XT = (-1)^{p_i(X)}\pi_ic_ic_{X\setminus \{i\}}T & = (-1)^{p_i(X)}c_{i+1}\pi_ic_{X\setminus \{i\}}T \\ & = (-1)^{p_i(X)}c_{i+1}c_{X\setminus \{i\}}\pi_iT =  c_{(X\setminus \{i\})\cup \{i+1\}}\pi_iT,
\end{align*}
where the last equality is due to the fact that, in this case, $p_{i+1}(X\setminus \{i\}) = p_i(X)$. The second line of each of (\ref{eqn:descent}), (\ref{eqn:ascent0}) and (\ref{eqn:ascentswap}) now follows from (\ref{eqn:0Hecke}).

If $i \notin X$ and $i+1 \in X$ then 
\begin{align*}
\pi_i(c_XT) & = (-1)^{p_{i+1}(X)}\pi_ic_{i+1}c_{X\setminus \{i+1\}}T  \\ 
                  & = (-1)^{p_{i+1}(X)}(c_i\pi_i + c_i -c_{i+1})c_{X\setminus \{i+1\}}T  \\ 
	 	 & = (-1)^{p_{i+1}(X)}(c_i\pi_i(c_{X\setminus \{i+1\}}T) + c_ic_{X\setminus \{i+1\}}T -c_{i+1}c_{X\setminus \{i+1\}}T) \\
		 & = c_{(X\setminus \{i+1\})\cup\{i\}}\pi_iT + c_{(X\setminus \{i+1\})\cup\{i\}}T  - c_{X}T,  
\end{align*}
where the last equality is due to the fact that in this case, $p_{i}(X\setminus \{i+1\}) = p_{i+1}(X)$. The third line of each of (\ref{eqn:descent}), (\ref{eqn:ascent0}) and (\ref{eqn:ascentswap}) now follows from (\ref{eqn:0Hecke}). In particular, for (\ref{eqn:descent}), if $\pi_iT=-T$ the first and second terms in the above expression cancel.

If $i , i+1 \in X$ then $p_i(X) = p_{i+1}(X\setminus \{i\})$ and $p_i(X\setminus\{i\}) = p_{i+1}(X\setminus \{i,i+1\})$, and 
\begin{align*}
\pi_i(c_XT) & = (-1)^{p_i(X)}(-1)^{p_{i+1}(X\setminus \{i\})}\pi_ic_ic_{i+1}c_{X\setminus \{i,i+1\}}T  \\ 
                  & = c_{i+1}(c_i\pi_i+c_i-c_{i+1})c_{X\setminus \{i,i+1\}}T \\
                  & = (-c_ic_{i+1}\pi_i-c_ic_{i+1}+1)c_{X\setminus \{i,i+1\}}T \\
                  & = -c_{X}\pi_iT-c_XT+c_{X\setminus \{i,i+1\}}T. 
        \end{align*}
The fourth line of each of (\ref{eqn:descent}), (\ref{eqn:ascent0}) and (\ref{eqn:ascentswap}) now follows from (\ref{eqn:0Hecke}). In particular, for (\ref{eqn:descent}), if $\pi_iT=-T$ the first and second terms in the above expression cancel.
\end{proof}

We note that relations equivalent to these are also observed in \cite{Li} for the $HCl_n(0)$-supermodules induced from projective covers of simple $H_n(0)$-modules.

\begin{remark}
The supermodule structure of $\widetilde{{\bf N}}_{\StdTab(D)}$ is given by letting the degree $0$ part of $\widetilde{{\bf N}}_{\StdTab(D)}$ be generated by those $c_XT$ with $|X|$ even and the degree $1$ part be generated by those $c_XT$ with $|X|$ odd. As can be seen from Proposition~\ref{prop:Dcrossrelations}, the degrees of the basis elements appearing in $\pi_i(c_XT)$ are equal to the degree of $c_XT$.
\end{remark}

We can represent a basis element $c_XT$ of $\widetilde{{\bf N}}_{\StdTab(D)}$ in tableau form by marking the entries of $T$ that are in $X$; see Example~\ref{ex:0HC} below.

\begin{example}\label{ex:0HC}
As in Example~\ref{ex:0H}, let $D = \tableau{  {\ } & {\ } \\ & {\ } }\, $, with reading order defined by reading the lower box first and then the higher boxes from left to right, and let
\[\StdTab(D) = \left\{ R = \tableau{ 1 & 3 \\ & 2 } \, , \,\,\, S = \tableau{ 2 & 3 \\ & 1 } \, ,\,\,\, T = \tableau{ 3 & 2 \\ & 1 } \right\}.\] 
The basis elements of $\widetilde{{\bf N}}_{\StdTab(D)}$ are 
\[\begin{array}{c @{\hskip2\cellsize}c@{\hskip2\cellsize}c@{\hskip2\cellsize}c@{\hskip2\cellsize}c@{\hskip2\cellsize}c@{\hskip2\cellsize}c@{\hskip2\cellsize}c}
\tableau{ 1 & 3 \\ & 2 } & \tableau{  1' & 3 \\ & 2 } & \tableau{ 1 & 3 \\ & 2' } & \tableau{ 1 & 3' \\ & 2 } & \tableau{  1' & 3 \\ & 2' } & \tableau{ 1' & 3'  \\ & 2 } & \tableau{ 1 & 3' \\ & 2' } & \tableau{ 1' & 3'  \\ & 2' } \\ \\
 \tableau{ 2 & 3  \\ & 1 } &  \tableau{ 2 & 3 \\ & 1' }  &  \tableau{ 2' & 3 \\ & 1 } &  \tableau{ 2 & 3' \\ & 1 } &  \tableau{ 2' & 3 \\ & 1' } &  \tableau{ 2 & 3' \\ & 1' } &  \tableau{ 2' & 3' \\ & 1 } &  \tableau{ 2' & 3' \\ & 1' } \\ \\ 
 \tableau{ 3 & 2 \\ & 1 } & \tableau{ 3 & 2 \\ & 1' } & \tableau{ 3 & 2' \\ & 1 } & \tableau{ 3' & 2 \\ & 1 } & \tableau{ 3 & 2' \\ & 1' } & \tableau{ 3' & 2 \\ & 1' } & \tableau{ 3' & 2' \\ & 1 } & \tableau{ 3' & 2' \\ & 1' } 
\end{array}\]
Using (\ref{eqn:descent}), (\ref{eqn:ascent0}) and (\ref{eqn:ascentswap}) respectively, we can compute, for example,
\begin{align*}
\pi_1(c_{\{1\}}R) & = \pi_1 \, \tableau{ 1' & 3 \\ & 2 } = - \, \tableau{ 1 & 3 \\ & 2' } \\
\pi_1(c_{\{1,2\}}T) & = \pi_1 \, \tableau{ 3 & 2' \\ & 1' } = - \, \tableau{ 3 & 2' \\ & 1' } +  \tableau{ 3 & 2 \\ & 1 }\\
\pi_1(c_{\{2,3\}}S) & = \pi_1 \, \tableau{ 2' & 3' \\ & 1 } = - \, \tableau{ 2' & 3' \\ & 1 } +  \tableau{  2 & 3' \\ & 1' } + \tableau{ 1' & 3' \\ & 2 }.
\end{align*}
\end{example}

Define a relation $\preceq$ on $\StdTab(D)$ by $S\preceq T$ if $S$ can be obtained from $T$ by applying a (possibly empty) sequence of the $\pi_i$ operators.

\begin{lemma}\label{lem:partialorder}
The relation $\preceq$ is a partial order on $\StdTab(D)$.
\end{lemma}
\begin{proof}
It is immediate from the definition that $\preceq$ is reflexive and transitive. For antisymmetry, recall that an \emph{inversion} in a permutation $\gamma$ is a pair of indices $r<s$ such that $\gamma(r)>\gamma(s)$. If $T, s_iT\in \StdTab(D)$, then the permutation $\rw(s_iT)$ has one more inversion than the permutation $\rw(T)$. Hence, if $S\preceq T$ and $S\neq T$, then $\rw(S)$ has strictly more inversions than $\rw(T)$ and therefore $T\npreceq S$. 
\end{proof}

Extend the partial ordering $\preceq$ to a total ordering $\preceq'$ arbitrarily, and suppose the elements of $\StdTab(D) = \{T_1, \ldots , T_m\}$ are ordered
$T_1 \preceq' T_2 \preceq' \cdots \preceq' T_m$. For each $1\le k \le m$, define subspaces $\widetilde{{\bf N}}_k$ of $\widetilde{{\bf N}}_{\StdTab(D)}$ by 
\[\widetilde{{\bf N}}_k = \rm{span}\{c_XT_j : X\subseteq [n], j\le k\}.\] 
Also define $\widetilde{{\bf N}}_0 = \{0\}$. The definition of the ordering $\preceq$, the formulas (\ref{eqn:descent}), (\ref{eqn:ascent0}), (\ref{eqn:ascentswap}), and the fact that $c_j(c_XT_k) = \pm c_{X\triangle \{j\}}T_k$, together imply that $\widetilde{{\bf N}}_{k}$ is a $HCl_n(0)$-sub-supermodule of $\widetilde{{\bf N}}_{\StdTab(D)}$ for each $0 \le k \le m$. We therefore have a filtration of $HCl_n(0)$-supermodules
\[\{0\}=\widetilde{{\bf N}}_0 \subset \widetilde{{\bf N}}_1 \subset \cdots \subset \widetilde{{\bf N}}_m = \widetilde{{\bf N}}_{\StdTab(D)}.\]

Given $1\le k \le m$, the quotient supermodule $\widetilde{{\bf N}}_k/\widetilde{{\bf N}}_{k-1}$ has basis $\{c_XT_k : X\subseteq [n]\}$. 

\begin{lemma}\label{lem:isom}
Let $T_k\in \StdTab(D)$ and let $\comp_n(\Des(T_k))=\alpha$. Then $\widetilde{{\bf N}}_k/\widetilde{{\bf N}}_{k-1}$ is isomorphic to ${\bf M}_\alpha$ as $HCl_n(0)$-supermodules.
\end{lemma}
\begin{proof}
Define a map $\psi:\widetilde{{\bf N}}_k/\widetilde{{\bf N}}_{k-1}\rightarrow {\bf M}_\alpha$ by $\psi(c_XT_k) = c_Xv_\alpha$, and extend by linearity. This map is evidently bijective and degree-preserving. 
To show it is a morphism of $HCl_n(0)$-supermodules, we first observe that for any $1\le j\le n$, we have 
\[\psi(c_jc_XT_k) = \psi((-1)^{p_j(X)} c_{X\triangle \{j\}}T_k) = (-1)^{p_j(X)} c_{X\triangle \{j\}}v_\alpha = c_jc_Xv_\alpha = c_j\psi(c_XT_k),\] 
where we recall $p_j(X)$ is the number of $x\in X$ such that $x<j$. Hence $\psi$ commutes with the action of the generators $c_j$. 

Showing $\psi$ commutes with the action of each $\pi_i$ for $1\le i \le n-1$ is done in two cases. By definition, $\Des(T_k) = \Des(\alpha)$. If $i\in \Des(T_k)$, then to see that $\psi(\pi_i(c_XT_k)) = \pi_i(c_Xv_\alpha)$ is a matter of comparing the cases in (\ref{eqn:descent}) with those in (\ref{eqn:MIdescent}). 

Now suppose $i\notin \Des(T_k)$. Then $\pi_i(c_XT_k)$ is determined by (\ref{eqn:ascent0}) or (\ref{eqn:ascentswap}). But we have $s_iT_k\in \widetilde{{\bf N}}_{k-1}$, and hence equal to zero in the quotient $\widetilde{{\bf N}}_k/\widetilde{{\bf N}}_{k-1}$. Therefore, in $\widetilde{{\bf N}}_k/\widetilde{{\bf N}}_{k-1}$, the equations (\ref{eqn:ascentswap}) become identical to those in (\ref{eqn:ascent0}). Hence for any $i\notin \Des(T_k)$, $\pi_i(c_XT_k)$ is determined by the cases in (\ref{eqn:ascent0}), and then it follows that $\psi(\pi_i(c_XT_k)) = \pi_i(c_Xv_\alpha)$ by comparing the cases in (\ref{eqn:ascent0}) with those in (\ref{eqn:MIascent}).
\end{proof}

Define the \emph{peak set} $\Peak(T)$ of $T\in \StdTab(D)$ to be $\Peak(\Des(T))$. 
Since ${\bf M}_\alpha$ is isomorphic to ${\bf M}_\beta$ as $HCl_n(0)$-supermodules when $\Peak(\alpha) = \Peak(\beta)$ \cite[Theorem 5.5]{BHT}, we have

\begin{lemma}\label{lem:dependonlyonpeak}
Let $T_k\in \StdTab(D)$. Then $\widetilde{{\bf N}}_k/\widetilde{{\bf N}}_{k-1}$ is isomorphic to ${\bf M}_{\comp_n(\Peak(T_k))}$ as $HCl_n(0)$-supermodules.
\end{lemma}

We can now determine the peak characteristic of $\widetilde{{\bf N}}_{\StdTab(D)}$ in terms of $\StdTab(D)$.

\begin{theorem}\label{thm:peakchar}
Let $D$ be a diagram in the plane with $n$ boxes and $\StdTab(D)$ an ascent-compatible subset of $\StdFill(D)$. Then 
\[\widetilde{ch}([\widetilde{{\bf N}}_{\StdTab(D)}]) = \sum_{T\in \StdTab(D)} K_{\comp_n(\Peak(T))}.\]
\end{theorem}
\begin{proof}
By Lemma~\ref{lem:dependonlyonpeak} and (\ref{eqn:peakcharM}) we have 
\begin{align*} \widetilde{ch}([\widetilde{{\bf N}}_{\StdTab(D)}]) & = \sum_{k=1}^m \widetilde{ch}([\widetilde{{\bf N}}_k/\widetilde{{\bf N}}_{k-1}])  \\ & = \sum_{k=1}^m \widetilde{ch}([{\bf M}_{\comp_n(\Peak(T_k))}]) = \sum_{T\in \StdTab(D)}K_{\comp_n(\Peak(T))}. \qedhere
\end{align*}
\end{proof}

One may similarly obtain a formula for the quasisymmetric characteristics of the $H_n(0)$-modules ${\bf N}_{\StdTab(D)}$ in terms of $\StdTab(D)$.
 
\begin{theorem}\label{thm:char}
Let $D$ be a diagram in the plane with $n$ boxes and $\StdTab(D)$ an ascent-compatible subset of $\StdFill(D)$. Then
\[ch([{\bf N}_{\StdTab(D)}]) = \sum_{T\in \StdTab(D)}F_{\comp_n(\Des(T))}.\]
\end{theorem}
\begin{proof}
Define a filtration of ${\bf N}_{\StdTab(D)}$ by
\[0={\bf N}_0 \subset {\bf N}_1 \subset \cdots \subset {\bf N}_m = {\bf N}_{\StdTab(D)},\]
where ${\bf N}_k = {\rm span}\{T_j : j \le k\}$. 
The successive quotients ${\bf N}_k/{\bf N}_{k-1}$ are one-dimensional, spanned by $T_k$, and satisfy 
\[\pi_iT_k = \begin{cases} -T_k & \mbox{ if } i\in \Des(T_k) \\
					0 & \mbox{ if } i\notin \Des(T_k)
\end{cases}\] 
Therefore ${\bf N}_k/{\bf N}_{k-1}$ is isomorphic to ${\bf F}_{\comp_n(\Des(T_k))}$ as $H_n(0)$-modules (see (\ref{eqn:irreps})), and the statement follows from (\ref{eqn:qsymchar}). 
\end{proof}

\begin{remark}
There is a surjective Hopf algebra homomorphism $\theta:\QSym \rightarrow \PQSym$ defined by $\theta(F_{\Des(\alpha)}) = K_{\Peak(\alpha)}$, originally introduced by Stembridge \cite[Theorem 3.1]{Stembridge:enriched}. Under this homomorphism, the image of the quasisymmetric characteristic $ch([{\bf N}_{\StdTab(D)}])$ is precisely the peak characteristic $\widetilde{ch}([\widetilde{{\bf N}}_{\StdTab(D)}])$.
\end{remark}

As mentioned in the introduction, an important motivation for the diagram modules paradigm arises from the fact that many important families of functions in $\Sym$, $\QSym$ and $\PQSym$ are defined as sums of fundamental quasisymmetric (or peak) functions associated to descent (or peak) sets of tableaux of varying shapes. Therefore, to obtain these functions as quasisymmetric or peak characteristics of diagram (super)modules, one needs only to check ascent-compatibility for those tableau families.  For all such families of tableaux that we are aware of in the context of $H_n(0)$-modules or $HCl_n(0)$-supermodules, the conditions for when exchanging entries $i$ and $i+1$ yields another tableau in the family depend only on the position (in $D$) of the box containing $i$ relative to the position of the box containing $i+1$ (e.g., weakly left of, strictly below, etc.).  
Hence the following lemma, although immediate from the definition of ascent-compatibility, will prove useful in the upcoming sections. 

\begin{lemma}\label{lem:position}
Let $D$ be a diagram and $\StdTab(D)$ a subset of $\StdFill(D)$. If the attacking/nonattacking status of ascents in elements of $\StdTab(D)$ is determined only by the relative position of the two boxes of $D$ associated to those ascents, then $\StdTab(D)$ is ascent-compatible.
\end{lemma}
\begin{proof}
This is equivalent to the attacking/nonattacking status of an ascent in position $(r,s)$ depending only on $r$ and $s$ (and not on $T\in \StdTab(D)$), so the statement follows from Definition~\ref{def:ascentcompatible}.
\end{proof}

%
\section{$0$-Hecke-Clifford supermodules for quasisymmetric Schur $Q$-functions}\label{sec:TypeCModule}
%

As our first application of the diagram modules framework, we consider a basis of the peak algebra known as the \emph{quasisymmetric Schur $Q$-functions}. Jing and Li introduced these functions in \cite{Jing.Li} and asked whether they can be given a representation-theoretic intepretation in terms of $0$-Hecke-Clifford supermodules. We answer this question affirmatively, and also prove that the supermodules we construct are cyclic, generated by a single tableau.

\subsection{Quasisymmetric Schur $Q$-functions}

The \emph{diagram} of a composition $\alpha$, denoted $D(\alpha)$, is the array of boxes with $\alpha_i$ boxes in row $i$, left-justified. The rows are numbered from bottom to top, e.g., if $\alpha = (4,2,1,3)$ then 
\[D(\alpha) = \tableau{ {\ } & {\ } & {\ } \\ {\ }  \\ {\ }  & {\ } \\ {\ } & {\ } & {\ } & {\ } }.\]

\begin{definition}\label{def:SPCT}
Let $\alpha$ be a peak composition of $n$. A \emph{standard peak composition tableau} of \emph{shape} $\alpha$ \cite{Jing.Li} is a standard filling of $D(\alpha)$ satisfying the following conditions.
\begin{enumerate}[leftmargin=1.8cm]
\item[(SPCT1)] Entries increase from left to right along each row.
\item[(SPCT2)] Entries increase from bottom to top in the first column.
\item[(SPCT3)] For every $1\le k \le n$, the subdiagram of $D(\alpha)$ consisting of the boxes with entries at most $k$ is the diagram of a peak composition.
\end{enumerate}
Let $\SPCT(\alpha)$ denote the set of all standard peak composition tableaux of shape $\alpha$.
\end{definition}

Given $T\in \SPCT(\alpha)$, the reading word $\rw(T)$ of $T$ is defined in \cite{Oguz} to be the word obtained by reading the entries in the columns of $T$ from bottom to top, starting with the leftmost column and proceeding rightwards. It follows that $i \in \Des(T)$ if and only if $i+1$ is either strictly left of $i$ in $T$ or in the same column as $i$ and below $i$.

\begin{example}\label{ex:SPCT}
Let $\alpha = (3,3,1)$. Then the elements of $\SPCT(\alpha)$, together with their descent sets, are
\[\begin{array}{c @{\hskip2\cellsize}c@{\hskip2\cellsize}c@{\hskip2\cellsize}c@{\hskip2\cellsize}c}
 \tableau{ 7 \\ 4 & 5 & 6 \\ 1 & 2 & 3 } & \tableau{ 6 \\ 4 & 5 & 7 \\ 1 & 2 & 3 }  &  \tableau{ 7 \\ 3 & 5 & 6 \\ 1 & 2 & 4 } & \tableau{ 6 \\ 3 & 5 & 7 \\ 1 & 2 & 4 } & \tableau{ 7 \\ 3 & 4 & 6 \\ 1 & 2 & 5 } \\ \\  
 \{3, 6\} & \{3, 5\} & \{2, 4, 6\} & \{2, 4, 5\} & \{2, 6\} \\ \\
 \tableau{ 6 \\ 3 & 4 & 7 \\ 1 & 2 & 5 }  &  \tableau{ 7 \\ 3 & 4 & 5 \\ 1 & 2 & 6 } &  \tableau{ 5 \\ 3 & 4 & 7 \\ 1 & 2 & 6 } & \tableau{ 6 \\ 3 & 4 & 5 \\ 1 & 2 & 7 } & \tableau{ 5 \\ 3 & 4 & 6 \\ 1 & 2 & 7 } \\ \\
  \{2, 5\} & \{2, 5, 6\} & \{2, 4\}& \{2, 5\}& \{2, 4, 6\} 
\end{array}\]
\end{example}

We are now ready to define the quasisymmetric Schur $Q$-functions $\widetilde{Q}_\alpha$. Jing and Li conjectured that these functions expand positively in the peak functions \cite{Jing.Li}, and Kantarc{\i} O{\u g}uz \cite{Oguz} resolved this conjecture in the affirmative, providing the following formula which we will use as our definition.
  
\begin{definition}\label{def:SchurQformula}\cite[Theorem 3.9]{Oguz}
Let $\alpha$ be a peak composition of $n$. Then 
\[\widetilde{Q}_\alpha = \sum_{T\in \SPCT(\alpha)}K_{\comp_n(\Peak(T))}.\]
\end{definition}

\begin{example}\label{ex:TypeCQSchur}
Let $\alpha = (3,3,1)$. Using Example~\ref{ex:SPCT} and noting that $\Peak(\{2,4,5\}) = \{2,4\}$ and $\Peak(\{2,5,6\}) = \{2,5\}$, we have 
\[\widetilde{Q}_{(3,3,1)} = K_{(3,3,1)} +  K_{(3,2,2)} +  2K_{(2,2,2,1)} + 2K_{(2,2,3)} + K_{(2,4,1)} +  3K_{(2,3,2)}.   \]
\end{example}

\subsection{The supermodules}

To construct $HCl_n(0)$-supermodules whose peak characteristics are the quasisymmetric Schur $Q$-functions, by Theorem~\ref{thm:peakchar} and Definition~\ref{def:SchurQformula} it is enough to verify that $\SPCT(\alpha)$ is ascent-compatible.

For any standard tableau $T$, we say an entry $i$ is \emph{immediately right (or left)} of an entry $j$ in $T$ if $i$ and $j$ are in the same row and $i$ is one column to the right (or left) of $j$ (in particular, $i$ and $j$ are adjacent entries in a row). Likewise, $i$ is \emph{immediately above (below)} $j$ if $i$ and $j$ are in the same column and $i$ is one row above (below) $j$.

\begin{lemma}\label{lem:SPCTdescents}
Let $\alpha$ be a peak composition of $n$ and $T\in \SPCT(\alpha)$. Then an ascent $i$ in $T$ is attacking if and only if $i+1$ is immediately right of $i$ in $T$.
\end{lemma}
\begin{proof}
It follows from the reading word that $i$ is an ascent in $T$ if and only if $i+1$ is either strictly right of $i$ in $T$, or above $i$ in the same column of $T$. 

If $i+1$ is immediately right of $i$ in $T$ then $s_iT\notin \SPCT(\alpha)$ by (SPCT1). Conversely, suppose $i$ is an ascent in $T$ and $i+1$ is not immediately right of $i$ in $T$. Then $i$ and $i+1$ cannot be in the same row of $T$ by (SPCT1), and thus (SPCT1) also holds for $s_iT$. Since $i+1$ is weakly right of $i$, the only way $i+1$ could be in the first column of $T$ is if $i$ is also in the first column, but this is impossible: (SPCT3) would then be violated because in the subdiagram corresponding to entries $1, \ldots , i+1$, the row containing $i$ and the row containing $i+1$ would both have length $1$. Hence $i+1$ cannot be in the first column of $T$, and thus (SPCT2) holds for $s_iT$. For (SPCT3), note that the subdiagram corresponding to entries $1, \ldots, k$ is the same for $T$ and $s_iT$ except when $k=i$. So the subdiagram of $s_iT$ corresponding to entries $1, \ldots, i-1$ is that of a peak composition, and then adding the entry $i$ still gives the diagram of a peak composition because the $i$ is placed in the box that contains $i+1$ in $T$, which is not in the first column. Hence $s_iT$ satisfies (SPCT3).
\end{proof}

Lemma~\ref{lem:SPCTdescents} establishes that the attacking/nonattacking status of ascents in $T\in \SPCT(\alpha)$ is determined by the relative position of the boxes involved. Therefore, Lemma~\ref{lem:SPCTascentcompatible} below follows immediately from Lemma~\ref{lem:position}.

\begin{lemma}\label{lem:SPCTascentcompatible}
Let $\alpha$ be a peak composition of $n$. Then $\SPCT(\alpha)$ is ascent-compatible.
\end{lemma}

Since $\SPCT(\alpha)$ is ascent-compatible, the diagram modules framework yields $HCl_n(0)$-supermodules $\widetilde{{\bf N}}_{\SPCT(\alpha)}$. Moreover, their peak characteristics are the quasisymmetric Schur $Q$-functions, thus answering the question of Jing and Li \cite{Jing.Li}.

\begin{theorem}\label{thm:QSQF}
Let $\alpha$ be a peak composition of $n$. Then
\[\widetilde{ch}([\widetilde{{\bf N}}_{\SPCT(\alpha)}]) = \widetilde{Q}_\alpha.\]
\end{theorem}
\begin{proof}
We have 
\[\widetilde{ch}([\widetilde{{\bf N}}_{\SPCT(\alpha)}]) = \sum_{T\in \SPCT(\alpha)}K_{\comp_n(\Peak(T))} = \widetilde{Q}_\alpha,\] where the first equality follows from Lemma~\ref{lem:SPCTascentcompatible} and Theorem~\ref{thm:peakchar} (letting $D=D(\alpha)$ and $\StdTab(D) = \SPCT(\alpha)$), and the second equality is Definition~\ref{def:SchurQformula}.
\end{proof}

\begin{remark}
Kantarc{\i} O{\u g}uz defines another family of tableaux called \emph{marked standard peak composition tableaux} and uses these to provide a formula for the expansion of $\widetilde{Q}_\alpha$ in the fundamental basis of $\QSym$ \cite[Theorem 3.7]{Oguz}. These tableaux are precisely the basis elements $c_XT$ of the $HCl_n(0)$-supermodule $\widetilde{{\bf N}}_{\SPCT(\alpha)}$ (interpreted as in Example~\ref{ex:0HC}), thus Theorem~\ref{thm:QSQF} and the diagram modules framework also provides a representation-theoretic interpretation of marked standard peak composition tableaux.
\end{remark}

\subsection{Cyclic structure}

Define $\widetilde{{\bf N}}_{\StdTab(D)}$ (or ${\bf N}_{\StdTab(D)}$) to be \emph{tableau-cyclic} if it is generated by a single $T\in \StdTab(D)$. 
We will prove $\widetilde{{\bf N}}_{\SPCT(\alpha)}$ is tableau-cyclic for all peak compositions $\alpha$. This provides a point of contrast with $HCl_n(0)$-supermodules we obtain in the next section that are not tableau-cyclic.

Given a peak composition $\alpha$, define the \emph{source tableau} $T^{\sup}_\alpha$ to be the element of $\SPCT(\alpha)$ whose entries in column $1$ are the first $\ell(\alpha)$ odd numbers, whose entries in column $2$ are the first $\ell(\alpha)-1$ or $\ell(\alpha)$  even numbers (depending on whether or not the last part of $\alpha$ is equal to $1$), and whose entries in subsequent columns are the remaining numbers, increasing consecutively up each column. 

\begin{example}
Let $\alpha = (3,4,3,2)$. Then 
\[T^{\sup}_\alpha = \tableau{ 7 & 8 \\ 5 & 6 & 11  \\ 3 & 4 & 10 & 12 \\ 1 & 2 & 9}.\] 
\end{example}

\begin{lemma}\label{lem:cyclic}
Let $\alpha \vDash n$ and $T\in \SPCT(\alpha)$. Then there is some sequence $\pi_{i_1}\ldots \pi_{i_r}$ where each $\pi_{i_j} \in \{\pi_1, \ldots , \pi_{n-1}\}$ such that $T= \pi_{i_1}\ldots \pi_{i_r}T_\alpha^{\sup}$.
\end{lemma}
\begin{proof}
Let $\mathfrak{b}$ be the earliest box of $D(\alpha)$ in reading order whose entry $j$ in $T$ is not the same as its entry $k$ in $T_\alpha^{\sup}$. We will show that we can obtain $T'\in \SPCT(\alpha)$ such that $T'$ agrees with $T$ on all entries in boxes earlier than $\mathfrak{b}$ and there is a sequence of $0$-Hecke operators taking $T'$ to $T$. Then, repeating this procedure as often as needed, we obtain $T^\star\in \SPCT(\alpha)$ such that $T^\star$ agrees with $T^{\sup}_\alpha$ on all boxes up to and including $\mathfrak{b}$, and $T$ is obtained by applying a sequence of $0$-Hecke operators to $T^\star$. We may then repeat this argument on the earliest box in which the entries of $T^\star$ and $T^{\sup}_\alpha$ disagree, and so forth.

Suppose first that $\mathfrak{b}$ is not in the first two columns of $D(\alpha)$. Then the boxes earlier than $\mathfrak{b}$ have entries $1, 2, \ldots k-1$ in both $T$ and $T_\alpha^{\sup}$. Hence $j>k$ and the entry $j-1 \ge k$ must appear later in reading order than $j$ in $T$. By (SPCT1) the entry $j-1$ is not immediately right of $j$ in $T$. Thus by Lemma~\ref{lem:SPCTdescents}, $T = \pi_{j-1}T'$, where $T'\in \SPCT(\alpha)$ is $T$ with entries $j-1$ and $j$ exchanged. In particular $T'$ has entry $j-1$ in box $\mathfrak{b}$. 

Next suppose $\mathfrak{b}$ is in the second column of $D(\alpha)$.  
The entries in the first column of $T$ are the first $\ell(\alpha)$ odd numbers. Therefore the lowest $\ell(\alpha)-1$ entries of the second column of $T$ must be the first $\ell(\alpha)-1$ even numbers, since if an entry $b$ in the second column was larger than a higher entry $a$ in the first column, (SPCT3) would be violated by the entries $1, \ldots , a$ and the row containing $b$. So if $\mathfrak{b}$ is in the second column of $T$, it must be the top box in the second column. We may then repeat the same argument as in the first case to find $T'$. 

Finally, suppose $\mathfrak{b}$ is in the first column of $D(\alpha)$.   
If $j-1$ is in the third column or later, then the tableau $T'$ obtained by swapping $j-1$ and $j$ in $T$ is in $\SPCT(\alpha)$, and $T = \pi_{j-1}T'$. Otherwise, $j-1$ cannot be in the first column of $T$ since this would violate (SPCT3) on the entries $1, \ldots , j$ and the row containing $j-1$, so $j-1$ is in the second column of $T$. In fact $T$ and $T_\alpha^{\sup}$ must agree on all boxes below $\mathfrak{b}$ in the first column and all boxes at least two boxes below the box immediately right of $\mathfrak{b}$ in the second column, since entries of these boxes in the second column are forced by the entries in the first column. Moreover, $j-1$ cannot occur in the row of $\mathfrak{b}$ or higher, by (SPCT1) and (SPCT2). This forces $j-1$ to be in the box immediately right of and below $\mathfrak{b}$. Since the entry $j$ of box $\mathfrak{b}$ in $T$ is at least three greater than the entry of the box immediately below it, the same is true for entry $j-1$ of the box immediately right of and below $\mathfrak{b}$. Hence $j-2$ cannot be in the first two columns of $T$, and thus we may exchange the entries $j-2$ and $j-1$, and then exchange the entries $j-1$ and $j$, to obtain $T'\in \SPCT(\alpha)$ with $T = \pi_{j-2}\pi_{j-1}T'$. The entry in box $\mathfrak{b}$ of $T'$ is $j-1$.
\end{proof}

\begin{theorem}\label{thm:SPCTcyclic}
Let $\alpha \vDash n$. Then $\widetilde{{\bf N}}_{\SPCT(\alpha)}$ is tableau-cyclic, generated by $T_\alpha^{\sup}$.
\end{theorem}
\begin{proof}
By Lemma~\ref{lem:cyclic}, every $T\in \SPCT(\alpha)$ may be obtained by application of a sequence of $\pi_i$'s to $T_\alpha^{\sup}$ (note $T_\alpha^{\sup}=c_\emptyset T_\alpha^{\sup}$, and thus we are in the first case of (\ref{eqn:ascentswap}) when applying this sequence). Then we can obtain $c_XT$ for any $X$ by applying the appropriate sequence of $c_j$'s.
\end{proof}

\begin{remark} The supermodules $\widetilde{{\bf N}}_{\SPCT(\alpha)}$ are not in general indecomposable. For example, $\SPCT(2,2,2,2) = \{T^{\sup}_{(2,2,2,2)}\}$, and it follows that $\widetilde{{\bf N}}_{\SPCT(2,2,2,2)}$ is isomorphic to ${\bf M}_{(2,2,2,2)}$. 
Since $\Peak(2,2,2,2) = \{2,4,6\}$ has three elements, by \cite[Corollary 5.6]{BHT} ${\bf M}_{(2,2,2,2)}$ decomposes into a sum of $2^{\lfloor\frac{3-1}{2}\rfloor}=2$ copies of the simple supermodule associated to $(2,2,2,2)$. 
\end{remark}

\section{A peak analogue of Young quasisymmetric Schur functions}
The $\SPCT$s used to construct the quasisymmetric Schur $Q$-function basis of $\PQSym$ are a subset of the \emph{standard immaculate tableaux} that define the dual immaculate basis of $\QSym$ \cite{BBSSZ:2}. Accordingly, the quasisymmetric Schur $Q$-functions may be regarded as a peak algebra analogue of the dual immaculate functions. We now consider the question of applying the diagram modules framework to find analogues in $\PQSym$ of other important bases of $\QSym$, as peak characteristics of $0$-Hecke-Clifford supermodules. 

The well-studied quasisymmetric Schur functions \cite{HLMvW11:QS} are an obvious next candidate. To facilitate comparison with the Schur $Q$-functions later, we will work instead with a variant called the \emph{Young quasisymmetric Schur functions} introduced in \cite{LMvWbook}. Young quasisymmetric Schur functions are obtained from quasisymmetric Schur functions by reversing both the indexing composition and the variable set.

\subsection{Diagram modules for Young quasisymmetric Schur functions}

\begin{definition}\label{def:SYCT}
Let $\alpha\vDash n$. A \emph{standard Young composition tableau} \cite{LMvWbook} of shape $\alpha$ is a standard filling $T$ of $D(\alpha)$ satisfying the following conditions.
\end{definition}

\begin{enumerate}[leftmargin=1.8cm]
\item[(SYCT1)] Entries of $T$ increase from left to right along rows.
\item[(SYCT2)] Entries of $T$ increase from bottom to top in the first column.
\item[(SYCT3)] Let $(c,r)$ denote the box in column $c$ and row $r$ and $T(c,r)$ its entry in $T$. If boxes $(c,r)$ and $(c+1,r')$ for $r'<r$ are in $D(\alpha)$ and $T(c,r)<T(c+1,r')$, then $(c+1,r)$ is in $D(\alpha)$ and $T(c+1,r)<T(c+1,r')$. 
\end{enumerate}

Pictorially, (SYCT3) states that for any three boxes arranged as below, if $a<b$ then $c<b$. We refer to these configurations of boxes as \emph{triples}.
\begin{equation}\label{eqn:SYCTtriples}
    \begin{array}{l}
        \tableau{ a & c } \\[-0.5\cellsize]  \hspace{1.5\cellsize} \vdots  \hspace{0.4\cellsize} \\ \tableau{  & b } 
    \end{array}
  \end{equation}
Let $\SYCT(\alpha)$ denote the set of all standard Young composition tableaux of shape $\alpha$. The \emph{descent set} of $T\in \SYCT(\alpha)$ is defined in \cite{LMvWbook} to be the entries $i$ such that $i+1$ is weakly left of $i$ in $T$. Then the Young quasisymmetric Schur function $\yqs_\alpha$ \cite{LMvWbook} is defined by
\begin{equation}\label{eqn:yqs}
\yqs_\alpha = \sum_{T\in \SYCT(\alpha)} F_{\comp_n(\Des(T))}.
\end{equation}

\begin{example}\label{ex:SYCT}
Let $\alpha = (2,4)$. The elements of $\SYCT(\alpha)$, along with their descent sets, are shown below.
\[\begin{array}{c@{\hskip2\cellsize}c@{\hskip2\cellsize}c@{\hskip2\cellsize}c@{\hskip2\cellsize}c}
\tableau{ 3 & 4 & 5 & 6 \\ 1 & 2 } & \tableau{ 2 & 3 & 5 & 6 \\ 1 & 4 } & \tableau{ 2 & 3 & 4 & 6 \\ 1 & 5 } & \tableau{ 2 & 3 & 4 & 5 \\ 1 & 6 }  \\ \\
\{ 2 \} & \{ 1, 3 \} & \{ 1, 4 \} &  \{ 1, 5 \} 
\end{array}\]
Therefore, $\yqs_{(2,4)} = F_{(2,4)} + F_{(1,2,3)} + F_{(1,3,2)} + F_{(1,4,1)}$.
\end{example}

To construct diagram $H_n(0)$-modules whose quasisymmetric characteristics are the Young quasisymmetric Schur functions, it suffices to confirm that the descent set of an $\SYCT$ arises from a reading word, and that $\SYCT(\alpha)$ is ascent-compatible. 

For $T\in \SYCT(\alpha)$, define $\rw(T)$ to be the word obtained by reading the entries of $T$ from top to bottom in the first column and then bottom to top in subsequent columns, reading the columns in order from left to right. For example, the reading word of the second tableau in Example~\ref{ex:SYCT} is $214356$.

\begin{lemma}\label{lem:SYCTdescent}
Let $\alpha\vDash n$ and $T\in \SYCT(\alpha)$. The descent set of $T$ obtained from $\rw(T)$ is precisely the entries $i$ such that $i+1$ is weakly left of $i$ in $T$.
\end{lemma}
\begin{proof}
If $i$ and $i+1$ are in different columns of $T$, then $i$ is read later than $i+1$ if and only if $i+1$ is left of $i$ in $T$. If $i$ and $i+1$ are both in the first column of $T$, then by (SYCT2) $i$ is below $i+1$, thus read later than $i+1$. If $i$ and $i+1$ are both in column $c$ of $T$ for $c>1$, then $i$ must be above $i+1$ since otherwise (SYCT3) is violated by $i$, $i+1$ and the entry immediately left of $i+1$ in $T$, and therefore $i$ is read later than $i+1$.  
\end{proof}

\begin{lemma}\label{lem:SYCTascentcompatible}
Let $\alpha \vDash n$. Then $\SYCT(\alpha)$ is ascent-compatible.
\end{lemma}
\begin{proof}
Let $T\in \SYCT(\alpha)$. We will show an ascent $i$ in $T$ is attacking if and only if $i+1$ is immediately right of $i$ in $T$. Then by Lemma~\ref{lem:position}, $\SYCT(\alpha)$ is ascent-compatible.

Let $i$ be an ascent in $T$. Then $i+1$ is strictly right of $i$ in $T$ by Lemma~\ref{lem:SYCTdescent}. If $i+1$ is immediately right of $i$ in $T$, then $i$ is attacking since $s_iT$ fails (SYCT1). If $i+1$ is not immediately right of $i$ in $T$, then $i+1$ is not in the same row as $i$, and moreover $i+1$ cannot be in the column immediately right of $i$ and a row below $i$, since then $i$ and $i+1$ would form two entries of a triple that would violate (SYCT3). Therefore $i$ and $i+1$ are not in the same row or column and do not both belong to any triple of entries. So the ordering of entries on rows, columns and triples is preserved when exchanging $i$ and $i+1$, and therefore $s_iT \in \SYCT(\alpha)$. Hence $i$ is nonattacking. 
\end{proof}

The next result now follows from (\ref{eqn:yqs}), Lemma~\ref{lem:SYCTascentcompatible} and Theorem \ref{thm:char}, letting $D=D(\alpha)$ and $\StdTab(D) = \SYCT(\alpha)$.

\begin{theorem}
Let $\alpha\vDash n$. Then 
$ch([{\bf N}_{\SYCT(\alpha)}]) = \yqs_\alpha.$
\end{theorem}

\subsection{Extracting a new basis of the peak algebra}

Via Lemma~\ref{lem:SYCTascentcompatible} and Theorem~\ref{thm:peakchar}, the diagram modules framework also immediately gives $HCl_n(0)$-supermodules $\widetilde{{\bf N}}_{\SYCT(\alpha)}$ satisfying
\[ \widetilde{ch}([\widetilde{{\bf N}}_{\SYCT(\alpha)}]) = \sum_{T\in \SYCT(\alpha)}K_{\comp_n(\Peak(T))}.\]
One can ask whether these functions form a basis of $\PQSym$, as $\alpha$ ranges over peak compositions. 
This is not the case: these functions are not linearly independent. However, after restricting to an appropriate subset of $\SYCT(\alpha)$, the peak characteristics of the resulting $HCl_n(0)$-supermodules do form a basis of $\PQSym$. 

\begin{definition}
Given a peak composition $\alpha\vDash n$, let the \emph{standard peak Young composition tableaux} $\SPYCT(\alpha)$ be the elements of $\SYCT(\alpha)$ such that for every $1\le k \le n$, the subdiagram of $D(\alpha)$ consisting of the boxes with entries at most $k$ is the diagram of a peak composition. 
\end{definition}

Note this extra condition is precisely the condition (SPCT3) required for $\SPCT(\alpha)$. We can use the other $\SYCT$ conditions to provide a simpler description of $\SPYCT(\alpha)$.

\begin{lemma}\label{lem:SYCT2peakcomp}
Let $\alpha$ be a peak composition of $n$ and let $T\in \SYCT(\alpha)$. Then $T\in \SPYCT(\alpha)$ if and only if entries in the second column of $T$ increase from bottom to top.
\end{lemma}
\begin{proof}
Suppose entries in the second column of $T$ increase from bottom to top. 
Then any triple involving an entry $a$ in the first column and an entry $b$ in the second column (see (\ref{eqn:SYCTtriples})) must have $a>b$, since otherwise by (SYCT3) we would need $c$ above $b$ in the second column with $c<b$. This implies that the entry immediately right of any entry $a$ in the first column is smaller than any entry above $a$, hence the subdiagram corresponding to the entries $1, \ldots , k$ is the diagram of a peak composition for any $k$.  

Conversely, suppose entries in the second column of $T$ do not increase from bottom to top. 
Consider the smallest entry $b$ in the second column such that the entry $c$ immediately above $b$ is smaller than $b$. Then by (SYCT3) the entry $a$ immediately left of $c$ must also be smaller than $b$. Then the subdiagram corresponding to the entries $1, \ldots , a$ is not the diagram of a peak composition, so $T\notin \SPYCT(\alpha)$.
\end{proof}

We use the same reading word for $T\in \SPYCT(\alpha)$ as for $T\in \SYCT(\alpha)$.

\begin{lemma}\label{lem:SYCT2compatible}
Let $\alpha$ be a peak composition of $n$. Then $\SPYCT(\alpha)$ is ascent-compatible.
\end{lemma}
\begin{proof}
By Lemma~\ref{lem:SYCTdescent}, given $T\in \SYCT(\alpha)$ we have $\pi_iT = s_iT$ only if $i+1$ is strictly right of $i$ in $T$. Hence the relative order of the entries in any column of $\pi_iT$ is the same as the relative order of entries in that column of $T$. In particular this is true for the second column of $T$, so by Lemma~\ref{lem:SYCT2peakcomp}, if $T\in \SPYCT(\alpha)$ and $\pi_iT=s_iT$ in ${\bf N}_{\SYCT(\alpha)}$ then $s_iT\in \SPYCT(\alpha)$. So the characterization of attacking descents in $\SPYCT(\alpha)$ is the same as that in $\SYCT(\alpha)$, meaning attacking is determined by relative position of boxes (Lemma~\ref{lem:SYCTascentcompatible}), and thus $\SPYCT(\alpha)$ is ascent-compatible by Lemma~\ref{lem:position}. 
\end{proof}

Let $\alpha$ be a peak composition of $n$. Lemma~\ref{lem:SYCT2compatible} and Theorem~\ref{thm:peakchar} allow us to define the \emph{peak Young quasisymmetric Schur function} $\widetilde{\yqs}_\alpha$ as the peak characteristic of $\widetilde{{\bf N}}_{\SPYCT(\alpha)}$, resulting in the following formula. 

\begin{theorem}\label{thm:peakcharYQS}
Let $\alpha$ be a peak composition of $n$. Then 
\[\widetilde{\yqs}_\alpha = \widetilde{ch}([\widetilde{{\bf N}}_{\SPYCT(\alpha)}]) = \sum_{T\in \SPYCT(\alpha)} K_{\comp_n(\Peak(T))}.\]
\end{theorem}

We now show the peak Young quasisymmetric Schur functions form a basis of $\PQSym$. Define a total ordering on the peak compositions of $n$ by $\alpha \succ \beta$ if for the smallest index $i$ such that $\alpha_i \neq \beta_i$, we have $\alpha_i > \beta_i$.  

\begin{lemma}\label{lem:order}
Let $\alpha$ be a peak composition of $n$. Then $\widetilde{\yqs}_\alpha = K_\alpha + \sum_{\beta}c_\beta K_\beta$, where $\beta$ ranges over peak compositions of $n$ and $c_\beta = 0$ whenever $\beta \succeq \alpha$.
\end{lemma}
\begin{proof}
Let $T^*$ be the standard filling of $D(\alpha)$ with entries $1, \ldots , \alpha_1$ in row $1$, entries $\alpha_1+1 , \ldots , \alpha_1+\alpha_2$ in row $2$, etc.  It is immediate that $T^*\in \SPYCT(\alpha)$. By definition, $\Des(T^*) = \{\alpha_1 , \alpha_1+\alpha_2, \ldots , \}$, and since $\alpha$ is a peak composition, $\alpha_1\neq 1$ and none of these values are consecutive. Hence $\Peak(T^*) = \Des(T^*)=\alpha$.

Now let $T \in \SPYCT(\alpha)$ and assume $\comp_n(\Peak(T)) \succeq \comp_n(\Peak(T^*))$. We claim that the smallest entry of $\Peak(T)$ is $j-1$, where $j$ is the smallest entry in $T$ that is not in row $1$ of $T$. To see this, note that $j$ must be in the leftmost box of row $2$ of $T$. Then since the subdiagrams of $D(\alpha)$ corresponding to the entries $1, \ldots , j$ and $1, \ldots , j+1$ in $T$ must be diagrams of peak compositions, we have $j-1\neq 1$ and $j+1$ is strictly right of $j$ in $T$. Therefore $j-1\neq 1$ is (the smallest) descent in $T$ and $j$ is not a descent in $T$, and thus $j-1\in \Peak(T)$ and there is no smaller element in $\Peak(T)$. Since $\comp_n(\Peak(T)) \succeq \comp_n(\Peak(T^*))$ we must have $j-1 \ge \alpha_1$, but there are only $\alpha_1$ boxes in row $1$ so we must have $j-1 = \alpha_1$. Therefore, the first row of $T$ consists of the entries $1, \ldots , \alpha_1$. 

Now repeat this argument for the second element of $\Peak(T)$, letting $j$ be the smallest entry in $T$ that is greater than $\alpha_1$ and not in row $2$ of $T$. Similarly, we have $j$ is in the leftmost box of row $3$ of $T$, $j-1 \neq \alpha_1+1$ and $j+1$ is strictly right of $j$ in $T$, so $j-1$ is the second element of $\Peak(T)$. Then $j-1 \ge \alpha_1+\alpha_2$ by assumption, which forces $j-1 = \alpha_1+\alpha_2$ since row $2$ has only $\alpha_2$ boxes. Therefore the second row of $T$ consists of the entries $\alpha_1+1 , \ldots , \alpha_1+\alpha_2$. Continuing thus, we find $T=T^*$. 
\end{proof}

\begin{theorem}\label{thm:basis}
The set $\{\widetilde{\yqs}_\alpha : \alpha \mbox{ is a peak composition}\}$ forms a basis of $\PQSym$. 
\end{theorem}
\begin{proof}
By Lemma~\ref{lem:order}, the transition matrix between the peak Young quasisymmetric Schur functions and the peak basis of $\PQSym$ is triangular.
\end{proof}

The peak Young quasisymmetric Schur functions do not expand positively in the quasisymmetric Schur $Q$-functions. However, computations suggest the quasisymmetric Schur $Q$-functions might expand positively in the peak Young quasisymmetric Schur functions. As further evidence for this, we can show that for any peak composition $\alpha$, the functions in the peak expansion of $\widetilde{\yqs}_\alpha$ are a subset of those in the peak expansion of $\widetilde{Q}_\alpha$.

\begin{proposition}
Let $\alpha\vDash n$ be a peak composition. Then $\widetilde{Q}_\alpha - \widetilde{\yqs}_\alpha$ expands positively in the peak basis.
\end{proposition}
\begin{proof}
Let $T\in \SPYCT(\alpha)$. Then $T$ satisfies (SPCT1), (SPCT2) and (SPCT3), so $T\in \SPCT(\alpha)$. Now, write $\Des_{\SPYCT}(T)$ for the descent set of $T$ thought of as an element of $\SPYCT(\alpha)$, and write $\Des_{\SPCT}(T)$ analogously. To show $T$ corresponds to the same peak function in both $\widetilde{Q}_\alpha$ and $\widetilde{\yqs}_\alpha$, it suffices to show $\Des_{\SPYCT}(T)=\Des_{\SPCT}(T)$. Since reading words for $\SPYCT(\alpha)$ and $\SPCT(\alpha)$ both read columns from left to right, if $i$ is in a different column to $i+1$ in $T$ then $i\in \Des_{\SPYCT}(T)$ if and only if $i\in \Des_{\SPCT}(T)$. Now suppose $i$ and $i+1$ are in the same column of $T$. This can't be the first column of $T$ since this would violate (SPCT3) on the entries $1, \ldots , i+1$. For $i$ and $i+1$ both in any other column of $T$, (SYCT3) forces $i$ to be above $i+1$, and therefore $i$ is in both $\Des_{\SPCT}(T)$ and $\Des_{\SPYCT}(T)$.
\end{proof}

\begin{remark}
In contrast to the $HCl_n(0)$-supermodules for quasisymmetric Schur $Q$-functions (Theorem~\ref{thm:SPCTcyclic}), the $HCl_n(0)$-supermodules $\widetilde{{\bf N}}_{\SYCT(\alpha)}$ and $\widetilde{{\bf N}}_{\SPYCT(\alpha)}$ are not tableau-cyclic in general. For either of these supermodules, the action of $\pi_i$ or $c_j$ on a tableau $T$ does not alter the relative order of entries in any column. Hence if $T, S\in \SYCT(\alpha)$ (or $\SPYCT(\alpha)$) and there is some column in $D(\alpha)$ in which the relative order of entries of $T$ and $S$ is different, then there is no element of $\SYCT(\alpha)$ or $\SPYCT(\alpha)$ from which both $T$ and $S$ can be obtained by applying sequences of $\pi_i$ and/or $c_j$.  
\end{remark}

\section{Connections to Schur $Q$-functions}

An appealing feature of (Young) quasisymmetric Schur functions is that they provide an particularly nice refinement of Schur functions. Therefore, it is natural to ask how the peak Young quasisymmetric functions relate to the peak algebra analogues of the Schur functions, namely, the Schur $Q$-functions. In this section we first use the diagram modules framework to construct $HCl_n(0)$-supermodules whose peak characteristics are the Schur $Q$-functions. We then establish an isomorphism of $HCl_n(0)$-supermodules that proves the basis of peak Young quasisymmetric Schur functions actually contains the Schur $Q$-functions. We also examine the $H_n(0)$-modules that arise alongside the $HCl_n(0)$-supermodules for Schur $Q$-functions via the diagram modules framework, and show their quasisymmetric characteristics are in fact Young quasisymmetric Schur functions.

\subsection{$0$-Hecke-Clifford supermodules for Schur $Q$-functions} 
Schur $Q$-functions form a basis for a Hopf subalgebra of the symmetric functions. They can be defined in terms of \emph{shifted} tableaux. 
A \emph{partition} $\lambda$ is a finite weakly decreasing sequence of positive integers; if the sequence is strictly decreasing then $\lambda$ is called a \emph{strict partition}. If the parts of $\lambda$ sum to $n$ we write $\lambda \vdash n$. 
The \emph{shifted diagram} of a strict partition $\lambda$ is the diagram $\ShD(\lambda)$ obtained from the usual (left-justified) diagram $D(\lambda)$ by shifting all boxes in the $i$th row $i-1$ units rightwards. 

\begin{definition}\label{def:SShT}
Let $\lambda$ be a strict partition. A \emph{standard shifted tableau} of shape $\lambda$ is a standard filling of $\ShD(\lambda)$ such that entries increase from left to right along each row and from bottom to top in each column. 
\end{definition}
Let $\SShT(\lambda)$ denote the set of all standard shifted tableaux of shape $\lambda$. Define the reading word $\rw(S)$ of $S\in \SShT(\lambda)$ by reading the entries of $S$ from left to right along rows, starting at the top row and proceeding downwards. It is then immediate that $i\in \Des(S)$ if and only if $i$ is strictly below $i+1$ in $S$. 

\begin{example}\label{ex:SchurQ}
Let $\lambda = (4,3,1)$. The elements of $\SShT(\lambda)$, along with their descent sets, are shown below.
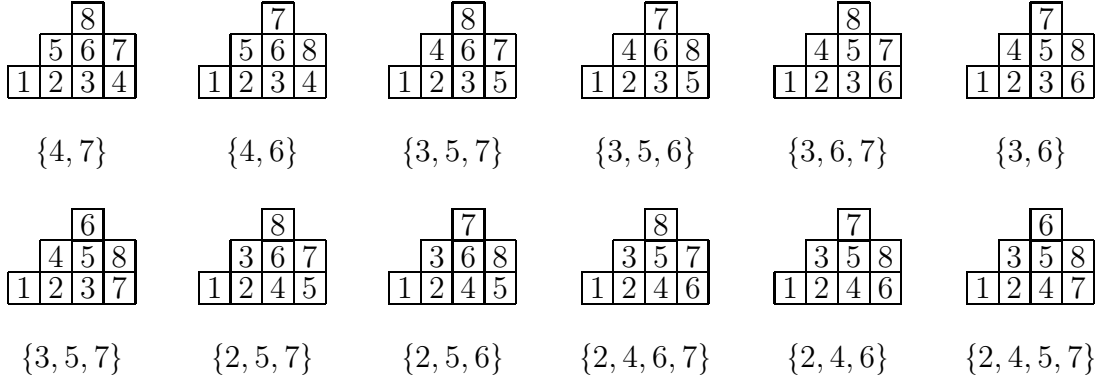
\[\begin{array}{c@{\hskip2\cellsize}c@{\hskip2\cellsize}c@{\hskip2\cellsize}c@{\hskip2\cellsize}c@{\hskip2\cellsize}c}
\tableau{ & & 8 \\  & 5 & 6 & 7 \\ 1 & 2 & 3 & 4 } & \tableau{ & & 7 \\  & 5 & 6 & 8 \\ 1 & 2 & 3 & 4 } & \tableau{ & & 8 \\  & 4 & 6 & 7 \\ 1 & 2 & 3 & 5 } & \tableau{ & & 7 \\  & 4 & 6 & 8 \\ 1 & 2 & 3 & 5 } & \tableau{ & & 8 \\  & 4 & 5 & 7 \\ 1 & 2 & 3 & 6 } & \tableau{ & & 7 \\  & 4 & 5 & 8 \\ 1 & 2 & 3 & 6 }  \\ \\  
\{4, 7\} & \{4, 6\} & \{3, 5, 7\} & \{3, 5, 6\} & \{3, 6, 7\} & \{3, 6\}  \\ \\
\tableau{ & & 6 \\  & 4 & 5 & 8 \\ 1 & 2 & 3 & 7 } &  \tableau{ & & 8 \\  & 3 & 6 & 7 \\ 1 & 2 & 4 & 5 } &  \tableau{ & & 7 \\  & 3 & 6 & 8 \\ 1 & 2 & 4 & 5 } &  \tableau{ & & 8 \\  & 3 & 5 & 7 \\ 1 & 2 & 4 & 6 } &  \tableau{ & & 7 \\  & 3 & 5 & 8 \\ 1 & 2 & 4 & 6 } &  \tableau{ & & 6 \\  & 3 & 5 & 8 \\ 1 & 2 & 4 & 7 }    
\\ \\ \{3, 5, 7\} & \{2, 5, 7\} & \{2, 5, 6\} &  \{2, 4, 6, 7\} & \{2, 4, 6\} & \{2, 4, 5, 7\} 
\end{array}\]
\end{example}

Given a strict partition $\lambda \vdash n$, the Schur $Q$-function $Q_\lambda$ can be defined by its expansion in the peak basis (\cite{Assaf:shifted, Stembridge:enriched}). 
\begin{equation}\label{eqn:SchurQintopeak}
Q_\lambda = \sum_{S\in \SShT(\lambda)}K_{\comp_n(\Peak(S))}.
\end{equation}
\begin{example}
Let $\lambda = (4,3,1)$. Using Example~\ref{ex:SchurQ} and noting that $\Peak(\{3, 5, 6\}) = \{3, 5\}$, $\Peak(\{3, 6, 7\}) = \{3, 6\}$ etc., we have
\begin{align*} 
Q_{(4,3,1)} & = K_{(4,3,1)} +  K_{(4,2,2)} + 2K_{(3,2,2,1)} + K_{(3,2,3)} +  2K_{(3,3,2)} \\   
& \quad +  K_{(2,3,2,1)} + K_{(2,3,3)} + 2K_{(2,2,2,2)} +  K_{(2,2,3,1)}.
\end{align*}
\end{example}

We can apply the diagram modules framework to standard shifted tableaux to obtain $HCl_n(0)$-supermodules whose peak characteristics are the Schur $Q$-functions.

\begin{lemma}\label{lem:SShTascentcompatible}
Let $\lambda\vdash n$ be a strict partition. Then $\SShT(\lambda)$ is ascent-compatible.
\end{lemma}
\begin{proof}
Let $S\in \SShT(\lambda)$ and let $i$ be an ascent in $S$. Then $i+1$ is weakly below $i$ in $S$. If $i+1$ is in the same row as $i$, then it is immediately right of $i$, and $s_iS\notin \SShT(\lambda)$. Conversely, if $i+1$ is strictly below $i$ in $S$, then $i$ and $i+1$ are in different rows and columns, and thus $s_iS\in \SShT(\lambda)$. 
Therefore an ascent $i$ in $S$ is attacking if and only if $i+1$ is immediately right of $i$ in $S$, and thus by Lemma~\ref{lem:position}, $\SShT(\lambda)$ is ascent-compatible.
\end{proof}

The next theorem, which we will subsequently use to show that Schur $Q$-functions are peak Young quasisymmetric Schur functions, provides a representation-theoretic interpretation of the expansion (\ref{eqn:SchurQintopeak}) of Schur $Q$-functions into the peak basis.

\begin{theorem}\label{thm:peakcharSchurQ}
Let $\lambda\vdash n$ be a strict partition. Then 
\[ch([{\bf N}_{\SShT(\lambda)}]) = \sum_{S\in \SShT(\lambda)}F_{\comp_n(\Des(S))} \qquad \mbox{ and } \qquad \widetilde{ch}([\widetilde{{\bf N}}_{\SShT(\lambda)}]) = Q_\lambda.\]
\end{theorem}
\begin{proof}
Follows from (\ref{eqn:SchurQintopeak}), Lemma~\ref{lem:SShTascentcompatible}, and Theorems~\ref{thm:char} and \ref{thm:peakchar}, letting $D=\ShD(\lambda)$ and $\StdTab(D) = \SShT(\lambda)$.
\end{proof}

\begin{remark}
Assaf uses \emph{marked standard (shifted) tableaux} to give a formula for the expansion of $Q_\lambda$ in the fundamental basis of $\QSym$ \cite[Proposition 4.2]{Assaf:shifted}. These tableaux are precisely the basis elements $c_XS$ of the $HCl_n(0)$-supermodule $\widetilde{{\bf N}}_{\SShT(\lambda)}$ (interpreted as in Example~\ref{ex:0HC}), thus Theorem~\ref{thm:peakcharSchurQ} and the diagram modules framework also provides a representation-theoretic interpretation of marked standard tableaux.
\end{remark}

\subsection{An isomorphism of $0$-Hecke-Clifford supermodules}

Observe that a strict partition $\lambda$ is also a peak composition, and therefore we have standard peak Young composition tableaux $\SPYCT(\lambda)$. 
Our next goal is to prove that for $\lambda$ a strict partition of $n$, $\widetilde{{\bf N}}_{\SShT(\lambda)}$ is isomorphic to $\widetilde{{\bf N}}_{\SPYCT(\lambda)}$ as $HCl_n(0)$-supermodules. 

\begin{lemma}\label{lem:SYCT2isSYCT}
Let $\lambda\vdash n$ be a strict partition. If $T\in \SPYCT(\lambda)$ or $T\in \SYCT(\lambda)$, then entries increase from bottom to top in every column of $T$. In particular, $\SPYCT(\lambda)=\SYCT(\lambda)$. 
\end{lemma}
\begin{proof}
Since $\SPYCT(\lambda)$ comprises those $T\in \SYCT(\lambda)$ for which entries in the second column increase from bottom to top (Lemma~\ref{lem:SYCT2peakcomp}), it  suffices to show this for $T\in \SYCT(\lambda)$. Suppose for a contradiction the statement is false, and consider the rightmost column of $T\in \SYCT(\lambda)$ in which there are entries $i>j$ with $j$ above $i$. Since $\lambda$ is strict, there exists a box immediately right of $i$, say with entry $k$. We have $k>i$ by (SYCT1), so $k>j$. Now, $j$ and $k$ form two boxes of a triple (\ref{eqn:SYCTtriples}), so by (SYCT3) there must exist a box immediately right of $j$, say with entry $\ell$, and $\ell <k$. But then $k$ and $\ell$ contradict our assumption that $i$ and $j$ were in the rightmost column in which entries don't increase from bottom to top.
\end{proof}

For $\lambda$ a strict composition, define a map $\rect : \SShT(\lambda) \rightarrow \StdFill(D(\lambda))$ by letting $\rect(S)$ be the filling of $D(\lambda)$ obtained by shifting all boxes in each row $i$ of $S$ by $i-1$ units leftwards, along with their entries. (Here, $D(\lambda)$ is the diagram of $\lambda$ interpreted as a composition, i.e., left-justified.)

\begin{lemma}\label{lem:despreserved}
The map $\rect$ is a bijection $\SShT(\lambda) \rightarrow \SPYCT(\lambda)$. Moreover, the descent set of $S\in \SShT(\lambda)$ is equal to the descent set of $\rect(S)\in \SPYCT(\lambda)$.
\end{lemma}
\begin{proof}
Let $S\in \SShT(\lambda)$. Since entries increase along rows of $S$, entries also increase along rows of $\rect(S)$, so (SYCT1) holds for $\rect(S)$. Define the $i$th \emph{diagonal} of $S$ to be the boxes $(i,1), (i+1,2), \ldots $ (i.e., the boxes of $S$ that become column $i$ in $\rect(S)$). Since entries of $S$ increase along rows and up columns, entries of $S$ must also increase from bottom to top along diagonals. Thus entries increase from bottom to top in all columns of $\rect(S)$, so (SYCT2) holds for $\rect(S)$. In any triple in $\rect(S)$, the boxes with entries $a,b$ (\ref{eqn:SYCTtriples}) correspond to boxes with entries $a,b$ in $S$ with $b$ strictly below and weakly left of $a$. Since entries increase along rows and up columns of $S$ we must have $a>b$, so (SYCT3) holds for $\rect(S)$. Therefore $\rect(S)\in \SPYCT(\lambda)$.

Conversely, let $T\in \SPYCT(\lambda)$ and let $S\in \StdFill(\ShD(\lambda))$ be such that $T=\rect(S)$. Since all columns of $T$ increase from bottom to top (Lemma~\ref{lem:SYCT2isSYCT}), we can not have $c<b$ in any triple, so we must have $a>b$ in all triples. This implies in particular that entries of $T$ must increase from bottom to top in all \emph{off-diagonals}, where the $i$th off-diagonal in $T$ is the boxes $(i,1), (i-1,2), \ldots$ (i.e., the boxes of $T$ that become column $i$ in $S$). Thus in $S$, entries increase from bottom to top in all columns, and clearly entries increase from left to right in all rows. Hence $S\in \SShT(\lambda)$, and $\rect$ is is a bijection $\SShT(\lambda) \rightarrow \SPYCT(\lambda)$.

Let $S\in \SShT(\lambda)$ and suppose $i$ is a descent in $S$. Then $i+1$ is strictly above $i$ in $S$, which implies $i+1$ must also be weakly left of $i$ in $S$, since if $i+1$ were strictly right of $S$ then the entry of $S$ in the row containing $i$ and column containing $i+1$ would have to be larger than $i$ and smaller than $i+1$, a contradiction. Therefore, in $\rect(S)$ we have $i+1$ weakly (in fact, strictly) left of $i$, since boxes in higher rows of $S$ move further left than boxes in lower rows under $\rect$. Hence $i$ is a descent in $\rect(S)$. Now suppose $i$ is an ascent in $S$. Then $i+1$ is weakly below $i$ in $S$, which also implies $i+1$ must be strictly right of $i$ in $S$. Then in $\rect(S)$ we have $i+1$ strictly right of $i$, since boxes in higher rows of $S$ move further left than boxes in lower rows under $\rect$. Hence $i$ is an ascent in $\rect(S)$. 
\end{proof}

\begin{example}\label{ex:yqsisom}
Let $\lambda = (4,3,1)$. The elements of $\SPYCT(\lambda)$, along with their descent sets, are shown below. Note these are the images of the elements of $\SShT(\lambda)$ (shown in Example~\ref{ex:SchurQ}) under $\rect$.
\[\begin{array}{c@{\hskip2\cellsize}c@{\hskip2\cellsize}c@{\hskip2\cellsize}c@{\hskip2\cellsize}c@{\hskip2\cellsize}c}
\tableau{ 8 \\  5 & 6 & 7 \\ 1 & 2 & 3 & 4 } & \tableau{  7 \\  5 & 6 & 8 \\ 1 & 2 & 3 & 4 } & \tableau{  8 \\  4 & 6 & 7 \\ 1 & 2 & 3 & 5 } & \tableau{  7 \\  4 & 6 & 8 \\ 1 & 2 & 3 & 5 } & \tableau{  8 \\  4 & 5 & 7 \\ 1 & 2 & 3 & 6 } & \tableau{  7 \\  4 & 5 & 8 \\ 1 & 2 & 3 & 6 }  \\ \\  
\{4, 7\} & \{4, 6\} & \{3, 5, 7\} & \{3, 5, 6\} & \{3, 6, 7\} & \{3, 6\}  \\ \\
\tableau{  6 \\  4 & 5 & 8 \\ 1 & 2 & 3 & 7 } &  \tableau{  8 \\  3 & 6 & 7 \\ 1 & 2 & 4 & 5 } &  \tableau{  7 \\  3 & 6 & 8 \\ 1 & 2 & 4 & 5 } &  \tableau{  8 \\  3 & 5 & 7 \\ 1 & 2 & 4 & 6 } &  \tableau{  7 \\  3 & 5 & 8 \\ 1 & 2 & 4 & 6 } &  \tableau{  6 \\  3 & 5 & 8 \\ 1 & 2 & 4 & 7 }    
\\ \\ \{3, 5, 7\} & \{2, 5, 7\} & \{2, 5, 6\} &  \{2, 4, 6, 7\} & \{2, 4, 6\} & \{2, 4, 5, 7\} 
\end{array}\]
\end{example}

\begin{theorem}\label{thm:moduleisom}
Let $\lambda$ be a strict partition of $n$. Then $\widetilde{{\bf N}}_{\SShT(\lambda)}$ is isomorphic as $HCl_n(0)$-supermodules to $\widetilde{{\bf N}}_{\SPYCT(\lambda)}$. 
\end{theorem}
\begin{proof}
Define a map $\phi:\widetilde{{\bf N}}_{\SShT(\lambda)} \rightarrow \widetilde{{\bf N}}_{\SPYCT(\lambda)}$ by setting $\phi(c_XS) = c_X\rect(S)$ and extending linearly. It is clear this map is degree-preserving and that for any $1\le j \le n$, we have $\phi(c_jc_XS) = c_j\phi(c_XS)$. To see $\phi(\pi_ic_XS) = \pi_i\phi(c_XS)$ for any $1\le i \le n-1$, note that $S$ and $\rect(S)$ have the same descent set by Lemma~\ref{lem:despreserved}, and moreover $i+1$ is immediately right of $i$ in $S$ if and only if $i+1$ is immediately right of $i$ in $\rect(S)$, so an ascent $i$ is attacking in $S$ if and only if $i$ is attacking in $\rect(S)$. 
\end{proof}

\begin{corollary}
The peak Young quasisymmetric Schur function basis of $\PQSym$ contains the Schur $Q$-functions.
\end{corollary}
\begin{proof}
By Theorem~\ref{thm:peakcharYQS}, we have $\widetilde{ch}([\widetilde{{\bf N}}_{\SPYCT(\lambda)}])=\widetilde{\yqs}_\lambda$. By Theorem~\ref{thm:moduleisom}, we have $\widetilde{ch}([\widetilde{{\bf N}}_{\SPYCT(\lambda)}]) = \widetilde{ch}([\widetilde{{\bf N}}_{\SShT(\lambda)}])$. By Theorem~\ref{thm:peakcharSchurQ}, we have $\widetilde{ch}([\widetilde{{\bf N}}_{\SShT(\lambda)}]) = Q_\lambda$. 
\end{proof}

The isomorphism $\phi$ in Theorem~\ref{thm:moduleisom} descends to an isomorphism of $H_n(0)$-modules between ${\bf N}_{\SShT(\lambda)}$ and ${\bf N}_{\SPYCT(\lambda)}$ and thus an isomorphism between ${\bf N}_{\SShT(\lambda)}$ and ${\bf N}_{\SYCT(\lambda)}$ by Lemma~\ref{lem:SYCT2isSYCT}.  This provides a further connection between Schur $Q$-functions and Young quasisymmetric Schur functions.

\begin{corollary}\label{cor:charsareyqs} 
For $\lambda$ a strict partition of $n$, the quasisymmetric characteristics of the $H_n(0)$-modules ${\bf N}_{\SShT(\lambda)}$ are the Young quasisymmetric Schur functions $\yqs_\lambda$. 
\end{corollary}

\begin{remark}
One can also use the diagram modules framework to define $H_n(0)$-modules on standard Young tableaux of partition shape $\lambda$, where $i$ is a descent of a standard Young tableau if $i$ is strictly below $i+1$. By a theorem of Gessel \cite{Gessel}, the Schur function for $\lambda$ is the sum of fundamental quasisymmetric functions indexed by descent sets of standard Young tableaux of shape $\lambda$. Therefore, the quasisymmetric characteristics of the diagram $H_n(0)$-modules on standard Young tableaux are the Schur functions. In comparison, by Corollary~\ref{cor:charsareyqs} the quasisymmetric characteristics of the diagram $H_n(0)$-modules on standard shifted tableaux are Young quasisymmetric Schur functions. 
\end{remark}

%
\section{A common framework for $0$-Hecke modules for bases of $\QSym$}
%

Recently, many families of $H_n(0)$-modules have been constructed so that their images under the quasisymmetric characteristic map are noteworthy bases of $\QSym$. This has been done in \cite{BBSSZ} for the dual immaculate functions introduced in \cite{BBSSZ:2},  in \cite{TvW:1} for the quasisymmetric Schur functions introduced in \cite{HLMvW11:QS}, in \cite{Searles:0Hecke} for the extended Schur functions introduced in \cite{Assaf.Searles:3},  in \cite{Bardwell.Searles} for the Young row-strict quasisymmetric Schur functions introduced in \cite{Mason.Niese}, and in \cite{NSvWVW:0Hecke} for the row-strict extended Schur functions and also for the row-strict dual immaculate functions introduced in \cite{NSvWVW:rowstrict}. In each case, the module is defined on the span of a family of standard tableaux whose underlying shapes are diagrams of compositions. More generally, $H_n(0)$-modules have been also defined on permuted versions of the tableaux for quasisymmetric Schur functions \cite{TvW:2}, Young quasisymmetric Schur functions \cite{CKNO:2}, and Young row-strict quasisymmetric Schur functions \cite{JKLO}.   

In this section, we define diagram $H_n(0)$-modules $\widehat{{\bf N}}_{\StdTab(D)}$ as a minor variant on ${\bf N}_{\StdTab(D)}$ that matches the conventions used in constructing all of the modules mentioned above. We then show that each of these modules is obtained as some diagram module $\widehat{{\bf N}}_{\StdTab(D)}$. We also consider the projective indecomposable $H_n(0)$-modules, which are interpreted in terms of standard tableaux of ribbon shape in \cite{Huang}, and show that these modules are obtained as certain ${\bf N}_{\StdTab(D)}$.  
Finally, we consider another general family of $H_n(0)$-modules known as \emph{weak Bruhat interval modules} \cite{JKLO}, and we show these modules are generalized by the families ${\bf N}_{\StdTab(D)}$ and $\widehat{{\bf N}}_{\StdTab(D)}$. 

\subsection{A variant on ${\bf N}_{\StdTab(D)}$}\label{subsec:variant}

The aforementioned modules in \cite{Bardwell.Searles,  BBSSZ, CKNO:2, JKLO, NSvWVW:0Hecke, Searles:0Hecke, TvW:1, TvW:2} are defined using the generators $\hat{\pi}_i = \pi_i+1$ of $H_n(0)$ instead of $\pi_i$. The first relation for the $0$-Hecke algebra ((\ref{eqn:qHecke}) at $q=0$) then becomes $\hat{\pi}_i^2 = \hat{\pi}_i$, while the other two relations remain the same: $\hat{\pi}_i \hat{\pi}_j = \hat{\pi}_j \hat{\pi}_i$ when $|i-j|\ge 2$, and $\hat{\pi}_i \hat{\pi}_{i+1} \hat{\pi}_i = \hat{\pi}_{i+1} \hat{\pi}_i\hat{\pi}_{i+1}$.

Roughly speaking, the role of descents and ascents becomes reversed under this change. To wit, in comparison to (\ref{eqn:irreps}), the action of $\hat{\pi}_i$ on simple $H_n(0)$-modules is 
\begin{align}\label{eqn:hatirreps}
\hat{\pi}_i(v_\alpha) = \begin{cases} 0 & \mbox{ if } i\in \Des(\alpha) \\
					       v_\alpha & \mbox{ if } i \notin \Des(\alpha).
					       \end{cases}
\end{align}

We can exactly recover the aforementioned modules in \cite{Bardwell.Searles, BBSSZ, CKNO:2, JKLO, NSvWVW:0Hecke, Searles:0Hecke, TvW:1, TvW:2} 
under the diagram modules framework by an analogous modification of the construction of ${\bf N}_{\StdTab(D)}$. For $T\in \StdTab(D)$, we say that a descent $i$ in $T$ is \emph{nonattacking} if $s_iT\in \StdTab(D)$ and \emph{attacking} otherwise. Define \emph{descent-compatibility} for $\StdTab(D)$ by replacing ``ascent'' with ``descent'' in Definition~\ref{def:ascentcompatible}, and replace (\ref{eqn:0Hecke}) by
\begin{equation}\label{eqn:hat0Hecke}
\hat{\pi}_iT = \begin{cases} T & \mbox{ if } i\notin \Des(T)  \\
                                        0 & \mbox{ if } i \in \Des(T), i \mbox{ is attacking } \\
                                        s_iT & \mbox{ if } i \in \Des(T), i \mbox{ is nonattacking. }
                                        \end{cases}
\end{equation}
That (\ref{eqn:hat0Hecke}) defines a $0$-Hecke action on the $\mathbb{C}$-span of $\StdTab(D)$ can be proved similarly to Theorem~\ref{thm:0Hecke}. Let $\widehat{{\bf N}}_{\StdTab(D)}$ denote the resulting $H_n(0)$-module. Then, following Lemma~\ref{lem:partialorder} and Theorem~\ref{thm:char}, we obtain 
\begin{equation}\label{eqn:hatchar}
ch([\widehat{{\bf N}}_{\StdTab(D)}]) = \sum_{T\in \StdTab(D)} F_{\comp_n(\Des(T))}.
\end{equation}

\subsection{Recovering known $H_n(0)$-modules as diagram modules}\label{subsec:recover}

For ease of exposition, we treat the modules mentioned at the beginning of this section in four separate cases.

\subsubsection{Modules for (row-strict) dual immaculate and extended Schur functions}

\begin{definition}\label{def:DIextend}
Let $\alpha\vDash n$. 

\begin{enumerate}
\item The \emph{standard immaculate tableaux} $\SIT(\alpha)$ \cite{BBSSZ} and \emph{standard row-strict immaculate tableaux} $\SRIT(\alpha)$ \cite{NSvWVW:0Hecke} are the standard fillings of $D(\alpha)$ that increase from left to right in each row and from bottom to top in the first column. 

\item The \emph{standard extended tableaux} $\SET(\alpha)$ \cite{Assaf.Searles:3, Searles:0Hecke}  and \emph{standard row-strict extended tableaux} $\SRET(\alpha)$ \cite{NSvWVW:0Hecke} are the standard fillings of $D(\alpha)$ that increase from left to right in each row and from bottom to top in each column. 
\end{enumerate}
\end{definition}

Although, e.g., $\SIT(\alpha)$ and $\SRIT(\alpha)$ have the same definition, we distinguish them because they will be defined to have different descent sets and thus give rise to distinct $H_n(0)$-modules. Also, since we number rows of $D(\alpha)$ from bottom to top, we note our description of $\SIT(\alpha)$ differs from that in \cite{BBSSZ} by a reflection in the $x$-axis.

Descent sets for each family are defined in terms of relative positions of boxes, and these can easily be seen to arise from a reading word.

\begin{definition}\label{def:DIextenddescents}
Let $\alpha\vDash n$. 
\begin{enumerate}
\item For $T\in \SIT(\alpha)$ or $T\in \SET(\alpha)$, $i\in \Des(T)$ if $i+1$ is strictly above $i$ in $T$ \cite[Definition 2.3]{BBSSZ}, \cite{Searles:0Hecke}. This is manifested by reading entries from left to right along rows, starting at the top row and proceeding downwards. 

\item For $T\in \SRIT(\alpha)$ or $T\in \SRET(\alpha)$, $i\in \Des(T)$ if $i+1$ is weakly below $i$ in $T$ \cite[Theorem 3.2 and Proposition 7.6]{NSvWVW:0Hecke}. This is manifested by reading entries from right to left along rows, starting at the bottom row and proceeding upwards. 
\end{enumerate}
\end{definition}

Note that for $\SET(\alpha)$, in \cite{Searles:0Hecke} the condition $i+1$ weakly left of $i$ is used instead. However this is easily seen to be equivalent to $i+1$ strictly above $i$, as pointed out in \cite[Theorem 7.4]{NSvWVW:0Hecke}.

\begin{theorem}\label{thm:previousmodulesSITSET}
Let $\alpha \vDash n$. For $\StdTab(D)$ equal to, respectively, $\SIT(\alpha)$, $\SET(\alpha)$, $\SRIT(\alpha)$, $\SRET(\alpha)$, the $H_n(0)$-module defined in, respectively, \cite[(4)]{BBSSZ} for dual immaculate functions, \cite[Theorem 3.5]{Searles:0Hecke} for extended Schur functions, \cite[(4)]{NSvWVW:0Hecke} for row-strict dual immaculate functions, and \cite[Lemma 7.2]{NSvWVW:0Hecke} for row-strict extended Schur functions, is precisely $\widehat{{\bf N}}_{\StdTab(D)}$.  
\end{theorem}
\begin{proof}
It follows from the descriptions of descents in Definition~\ref{def:DIextenddescents}(1) that a descent $i$ is attacking in $T\in \SIT(\alpha)$ if and only if $i$ and $i+1$ are both in the first column of $T$, and that a descent $i$ is attacking in $T\in \SET(\alpha)$ if and only if $i$ and $i+1$ are both in the same column of $T$. By \cite[Lemma 4.4]{NSvWVW:0Hecke}, a descent $i$ is attacking in $T\in \SRIT(\alpha)$ or $\SRET(\alpha)$ if and only if $i+1$ is immediately right of $i$ in $T$. 

Now, comparing the structure of the modules defined on $\SIT(\alpha)$ in \cite[(4)]{BBSSZ} and on $\SET(\alpha)$ in \cite[Theorem 3.5]{Searles:0Hecke} to this classification of attacking and nonattacking descents for $\SIT(\alpha)$ and $\SET(\alpha)$, and comparing the structure of the modules defined on $\SRIT(\alpha)$ and $\SRET(\alpha)$ in \cite[(4), Lemma 7.2]{NSvWVW:0Hecke} to the classification \cite[Lemma 4.4(2)]{NSvWVW:0Hecke} of attacking and nonattacking descents for $\SRIT(\alpha)$ and $\SRET(\alpha)$, we have that all four of these modules have structure given by
\begin{equation}\label{eqn:firstcase}
\hat{\pi}_iT = \begin{cases} T & \mbox{ if } i\notin \Des(T)  \\
                                        0 & \mbox{ if } i \in \Des(T), i \mbox{ is attacking } \\
                                        s_iT & \mbox{ if } i \in \Des(T), i \mbox{ is nonattacking. }
                                        \end{cases}
\end{equation}

On the other hand, since attacking is determined by relative position of boxes for each of $\SIT(\alpha)$, $\SET(\alpha)$, $\SRIT(\alpha)$, $\SRET(\alpha)$, these families of tableaux are all descent-compatible by Lemma~\ref{lem:position} (which holds for descents as well as ascents). Thus, for $\StdTab(D)$ equal to, respectively, $\SIT(\alpha)$, $\SET(\alpha)$, $\SRIT(\alpha)$, $\SRET(\alpha)$, we have $\widehat{{\bf N}}_{\StdTab(D)}$ is an $H_n(0)$-module. 
Since (\ref{eqn:firstcase}) is the same as (\ref{eqn:hat0Hecke}), these modules $\widehat{{\bf N}}_{\StdTab(D)}$ are the same as those \cite{BBSSZ, NSvWVW:0Hecke, Searles:0Hecke} in each case.
\end{proof}

\subsubsection{Modules for (Young row-strict) quasisymmetric Schur functions}\label{subsubsec:quasisymmetricSchur}

\begin{definition}\label{def:SRCTSYRT}
Let $\alpha\vDash n$.  

\begin{enumerate}
\item The \emph{standard reverse composition tableaux} $\SRCT(\alpha)$ \cite{TvW:1} are the standard fillings $T$ of $D(\alpha)$ that decrease from left to right in each row, increase from bottom to top the first column and satisfy the condition \emph{(triple rule)} that if boxes $(c,r)$ and $(c+1,r')$ for $r<r'$ are in $D(\alpha)$ and $T(c,r)> T(c+1,r')$, then $(c+1,r)$ is in $D(\alpha)$ and $T(c+1,r)>T(c+1,r')$.      

\item The \emph{standard Young row-strict composition tableaux} $\SYRT(\alpha)$ \cite{Mason.Niese} are the same as $\SYCT(\alpha)$ (Definition~\ref{def:SYCT}). 
\end{enumerate}
\end{definition}
Since we number rows of $D(\alpha)$ from bottom to top, our definition of $\SRCT(\alpha)$ below differs from that in \cite{TvW:1} by a reflection in the $x$-axis.

Descent sets for each family are defined in terms of relative positions of boxes, and can be seen to arise from a reading word.

\begin{definition}\label{def:SRCTSYRTdescents}
Let $\alpha\vDash n$.
\begin{enumerate}
\item For $T\in \SRCT(\alpha)$, $i\in \Des(T)$ if $i+1$ is weakly right of $i$ in $T$ \cite{TvW:1}. This is manifested by reading entries from bottom to top in columns, starting at the rightmost column and proceeding leftwards, except the first column (the last to be read) is read from top to bottom.

\item For $T\in \SYRT(\alpha)$, $i\in \Des(T)$ if $i+1$ is strictly right of $i$ in $T$ \cite{Bardwell.Searles}. This is manifested by reading entries from top to bottom in columns, starting at the rightmost column and proceeding leftwards, except the first column (the last to be read) is read from bottom to top.
\end{enumerate}
\end{definition}

\begin{lemma}\label{lem:attackingdescents}
Let $\alpha\vDash n$.
\begin{enumerate}
\item A descent $i$ is attacking in $T\in \SRCT(\alpha)$ if and only if either $i+1$ and $i$ are both in the first column, or $i+1$ is below $i$ in a column other than the first column, or $i+1$ is strictly above $i$ in the column immediately right of $i$.
\item A descent $i$ is attacking in $T\in \SYRT(\alpha)$ if and only if $i+1$ is immediately right of $i$.
\end{enumerate}
\end{lemma}
\begin{proof}
The statement of (1) is identical to the definition of attacking in \cite[Definition 3.1]{TvW:1} (reflected in the $x$-axis). In \cite[Lemma 3.7]{TvW:1} it is shown that $s_iT \in \SRCT(\alpha)$ if $i$ is nonattacking. For the converse, suppose $i$ is attacking. If $i+1$ is above $i$ in the first column, then certainly $s_iT\notin \SRCT(\alpha)$. For any $T\in \SRCT(\alpha)$, one cannot have $i+1$ above $i$ in a column other than the first, as then the entry immediately left of $i$ would have to be greater than $i$ since entries decrease along rows, violating the triple rule. Hence if $i+1$ is below $i$ in a column of $T$ other than the first, then $s_iT\notin \SRCT(\alpha)$. Finally if $i+1$ is strictly above $i$ in the column immediately right of $i$ in $T$, then $s_iT\notin \SRCT(\alpha)$ since the triple rule is violated on $i$, $i+1$ and the entry immediately right of $i+1$ in $s_iT$. 

For (2), clearly such a descent is attacking, and any other descent is nonattacking by \cite[Lemma 3.2]{Bardwell.Searles}.
\end{proof}

\begin{theorem}\label{thm:previousmodulesSPCTSYRT}
Let $\alpha\vDash n$. For $\StdTab(D)$ equal to, respectively, $\SRCT(\alpha)$ or $\SYRT(\alpha)$, the $H_n(0)$-module defined in \cite[(3.1)]{TvW:1} for quasisymmetric Schur functions or, respectively, \cite[Theorem 3.3]{Bardwell.Searles} for Young row-strict quasisymmetric Schur functions, is precisely $\widehat{{\bf N}}_{\StdTab(D)}$.  
\end{theorem}
\begin{proof}
This follows from comparing the definitions of the modules in \cite[(3.1)]{TvW:1} and \cite[Theorem 3.3]{Bardwell.Searles} to the classifications of attacking and nonattacking descents in Lemma~\ref{lem:attackingdescents}, and a similar argument to Theorem~\ref{thm:previousmodulesSITSET}. 
\end{proof}

\subsubsection{Permuted variants}
The tableau families $\SRCT(\alpha)$ and $\SYRT(\alpha)$ are generalised to, respectively, $\SRCT^\sigma(\alpha)$ in \cite[Section 2.4]{TvW:2} and $\SYRT^\sigma(\alpha)$ in \cite[Definition 4.10]{JKLO}, where $\sigma$ is a fixed permutation of $[\ell(\alpha)]$. The definitions of $\SRCT^\sigma(\alpha)$ and $\SYRT^\sigma(\alpha)$ are the same as those for $\SRCT(\alpha)$ and $\SYRT(\alpha)$ respectively (Definition~\ref{def:SRCTSYRT}), except that entries must now follow the same relative order as $\sigma$ from bottom to top in the first column of $T$, instead of increasing from bottom to top. Then $H_n(0)$-modules on $\SRCT^\sigma(\alpha)$ and $\SYRT^\sigma(\alpha)$ are defined in \cite[(3.1)]{TvW:2} and \cite[(4.3)]{JKLO} respectively. The classifications of descents and attacking descents are the same for $\SRCT^\sigma(\alpha)$ and $\SYRT^\sigma(\alpha)$ as for $\SRCT(\alpha)$ and $\SYCT(\alpha)$ respectively. Then, changing the reading word for $\SRCT(\alpha)$ so that boxes in the first column are read in the reverse order to $\sigma$, and the reading word for $\SYRT(\alpha)$ so that boxes in the first column are read in the order of $\sigma$, and following the arguments in Subsection~\ref{subsubsec:quasisymmetricSchur}, the modules on $\SRCT^\sigma(\alpha)$ in \cite{TvW:2} and $\SYRT^\sigma(\alpha)$ in \cite{JKLO} are precisely the diagram modules $\widehat{{\bf N}}_{\SRCT^\sigma(\alpha)}$ and $\widehat{{\bf N}}_{\SYRT^\sigma(\alpha)}$ respectively. Furthermore, an $H_n(0)$-module is defined on $\SYCT^\sigma(\alpha)$ in \cite[(4.2)]{CKNO:2}; this can similarly be seen to be equal to $\widehat{{\bf N}}_{\SYCT^\sigma(\alpha)}$, once the reading word of $T\in \SYCT^\sigma(\alpha)$ is is defined to be that obtained by reading entries of $T$ in reverse order to $\sigma$ in the first column, and then from bottom to top in subsequent columns, reading the columns in order from left to right. In short, we have the following.

\begin{theorem}\label{thm:previousmodulespermuted}
Let $\alpha\vDash n$. For $\StdTab(D)$ equal to, respectively, $\SRCT^\sigma(\alpha)$, $\SYRT^\sigma(\alpha)$, $\SYCT^\sigma(\alpha)$ the $H_n(0)$-module defined in, respectively, \cite[(3.1)]{TvW:2}, \cite[(4.3)]{JKLO}, \cite[(4.2)]{CKNO:2}, is precisely $\widehat{{\bf N}}_{\StdTab(D)}$.
\end{theorem}

\subsubsection{Projective indecomposable $0$-Hecke modules}

Huang \cite{Huang} interpreted the projective indecomposable $H_n(0)$-modules in terms of \emph{ribbon tableaux}. 
Given a composition $\alpha\vDash n$, the \emph{ribbon diagram} of $\alpha$ is the diagram whose $i$th row from the bottom consists of $\alpha_i$ boxes, such that the leftmost box in each row is immediately above the rightmost box in the row below. The \emph{standard ribbon tableaux} of shape $\alpha$, denoted $\Rib(\alpha)$, are the standard fillings of the ribbon diagram of $\alpha$ such that entries increase from left to right along rows and from top to bottom in columns.

\begin{example}
The standard ribbon tableaux of shape $\alpha=(2,2)$ are
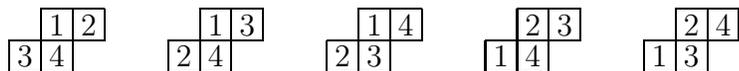
\[\begin{array}{c@{\hskip2\cellsize}c@{\hskip2\cellsize}c@{\hskip2\cellsize}c@{\hskip2\cellsize}c@{\hskip2\cellsize}c}
\tableau{ & 1 & 2  \\ 3 & 4 } & \tableau{ & 1 & 3  \\ 2 & 4 } &\tableau{ & 1 & 4  \\ 2 & 3 } &\tableau{ & 2 & 3  \\ 1 & 4 } &\tableau{ & 2 & 4  \\ 1 & 3 } 
\end{array}\]
\end{example}

Given a composition $\alpha$, the projective indecomposable $H_n(0)$-module ${\bf P}_\alpha$ is defined \cite[(3.3)]{Huang} to be the $\mathbb{C}$-span of $\Rib(\alpha)$, with $0$-Hecke action given by 
\begin{equation}\label{eqn:projective}
\pi_iT = \begin{cases} -T & \mbox{ if } i \mbox{ is strictly above } i+1  \\
                                        0 & \mbox{ if } i \mbox{ is in the same row as } i+1 \\
                                        s_iT & \mbox{ if } i \mbox{ is strictly below } i+1
                                        \end{cases}
\end{equation}
for each $T\in \Rib(\alpha)$.

\begin{theorem}
Let $\alpha \vDash n$ and define a reading order on $\Rib(\alpha)$ by reading entries from left to right along rows, starting at the bottom row and proceeding upwards. Then the module ${\bf P}_\alpha$ is precisely the diagram module ${\bf N}_{\Rib(\alpha)}$.
\end{theorem}
\begin{proof}
Let $T\in \Rib(\alpha)$. The reading order implies $i\in \Des(T)$ if and only if $i$ is strictly above $i+1$ in $T$. If $i$ is an ascent in $T$, then $i$ and $i+1$ are in different columns and thus $i$ is attacking if and only if $i$ and $i+1$ are in the same row. So by  Lemma~\ref{lem:position}, ${\bf N}_{\Rib(\alpha)}$ is an $H_n(0)$-module. Using this classification of descents and (attacking) ascents in $T$, and comparing (\ref{eqn:projective}) and (\ref{eqn:0Hecke}), we find that ${\bf N}_{\Rib(\alpha)}$ is precisely ${\bf P}_\alpha$.
\end{proof}

\subsection{Intervals in left weak Bruhat order}

The modules in \cite{Bardwell.Searles, BBSSZ, Searles:0Hecke, TvW:1, TvW:2} have also been interpreted by Jung, Kim, Lee and Oh \cite{JKLO} as sums of \emph{weak Bruhat interval modules}. These $H_n(0)$-modules, which can also be realised in terms of $H_n(0)$-modules on Yang-Baxter intervals \cite{Hivert.Novelli.Thibon}, are defined on intervals in the left weak Bruhat order on the symmetric group $\mathfrak{S}_n$. 

For $\gamma\in \mathfrak{S}_n$, define $i$ to be a descent of $\gamma$ if $i$ is to the right of $i+1$ in the one-line notation for $\gamma$, and let $\Des(\gamma)$ denote the set of descents of $\gamma$. Note this is the same as the definition (\ref{eqn:rwdescent}) of descent for (the reading word of) a standard tableau. If $s_i$ is a simple transposition, then $s_i\gamma$ is the permutation obtained by exchanging $i$ and $i+1$ in the one-line notation for $\gamma$. The \emph{left weak Bruhat order} $\preceq_L$ on $\mathfrak{S}_n$ can be defined by its covering relations: $\delta$ covers $\gamma$ in $\preceq_L$ if $s_i\gamma=\delta$ for some $i\notin \Des(\gamma)$.

Let $\sigma \preceq_L \rho$ in $\mathfrak{S}_n$, and let $[\sigma, \rho]_L$ denote the corresponding interval in $\preceq_L$, i.e., the set of permutations $\gamma \in \mathfrak{S}_n$ such that $\sigma \preceq_L \gamma \preceq_L \rho$. In \cite[Definition 3.1]{JKLO}, Jung, Kim, Lee and Oh define $H_n(0)$-modules $B(\sigma, \rho)$ and $\overline{B}(\sigma, \rho)$ on the $\mathbb{C}$-span of $[\sigma, \rho]_L$. Letting $\gamma\in [\sigma, \rho]_L$, the structure of $B(\sigma, \rho)$ is given by
\begin{equation}\label{eqn:B}
\hat{\pi}_i\gamma = \begin{cases} \gamma & \mbox{ if } i\in \Des(\gamma)  \\
                                        0 & \mbox{ if } i \notin \Des(\gamma), s_i\gamma \notin [\sigma, \rho]_L  \\
                                        s_i\gamma & \mbox{ if } i \notin \Des(\gamma), s_i\gamma \in [\sigma, \rho]_L 
                                        \end{cases}
\end{equation}
and the structure of $\overline{B}(\sigma, \rho)$ is given by
\begin{equation}\label{eqn:barB}
\pi_i\gamma = \begin{cases} -\gamma & \mbox{ if } i\in \Des(\gamma)  \\
                                        0 & \mbox{ if } i \notin \Des(\gamma), s_i\gamma \notin [\sigma, \rho]_L  \\
                                        s_i\gamma & \mbox{ if } i \notin \Des(\gamma), s_i\gamma \in [\sigma, \rho]_L. 
                                        \end{cases}
\end{equation}

Given an interval $[\sigma, \rho]_L$ in left weak Bruhat order on $\mathfrak{S}_n$, one can define a diagram $D$ and $\StdTab(D)$ such that the reading words of $\StdTab(D)$ are precisely $[\sigma, \rho]_L$. For example, let $D$ be a single row of $n$ boxes with reading order from left to right, and let $\StdTab(D)$ comprise the fillings obtained by writing the one-line notation for each $\gamma \in [\sigma, \rho]_L$ in the boxes of $D$ from left to right. 

\begin{lemma}\label{lem:Bruhatascent}
Let $\sigma \preceq_L \rho$ in $\mathfrak{S}_n$. For any choice of $D$ and $\StdTab(D)$ such that $\{\rw(T) : T\in \StdTab(D)\}$ is equal to $[\sigma, \rho]_L$,  $\StdTab(D)$ is ascent-compatible. 
\end{lemma}
\begin{proof}
Let $\gamma \in \mathfrak{S}_n$. For positions $r<s$, we say $(r,s)$ is an \emph{ascent pair} in $\gamma$ if $\gamma(r) < \gamma(s)$, and a \emph{descent pair} in $\gamma$ if $\gamma(r)>\gamma(s)$. 
We will prove that $\gamma\preceq_L \rho$ if and only if every ascent pair in $\rho$ is also an ascent pair in $\gamma$. From this it follows that for any $\gamma\preceq_L \rho$, if $\gamma(r) = i$ and $\gamma(s)=i+1$ then $s_i\gamma \preceq_L \rho$ if and only if $\rho(r)>\rho(s)$, and therefore whether an ascent in a permutation in $[\sigma, \rho]_L$ is attacking or not depends only on its position $(r,s)$, ensuring ascent-compatibility.

Certainly if $(r,s)$ is an ascent pair in $\rho$ and a descent pair in $\gamma$, then we cannot have $\gamma \preceq_L \rho$, since if $\delta\preceq_L s_i\delta$ for some $\delta\in \mathfrak{S}_n$ and some $i$, then applying $s_i$ to $\delta$ cannot change any descent pair to an ascent pair. 
 
For the converse, suppose every ascent pair in $\rho$ is also an ascent pair in $\gamma$. We will prove $\gamma\preceq_L \rho$ by induction on $\ell(\rho) - \ell(\gamma)$, where $\ell(\gamma)$ is the number of descent pairs in $\gamma$. 
If $\ell(\rho) - \ell(\gamma) = 0$ then the ascent pairs in $\rho$ are exactly the ascent pairs in $\gamma$. A list of ascent pairs determines a permutation, so we have $\gamma=\rho$ and therefore $\gamma\preceq_L \rho$. 
Suppose inductively that $\gamma\preceq_L \rho$ whenever every ascent pair in $\rho$ is also an ascent pair in $\gamma$ and $\ell(\rho) - \ell(\gamma)<k$ for some $k\ge 1$. Now let $\gamma$ be such that every ascent pair in $\rho$ is also an ascent pair in $\gamma$ and $\ell(\rho) - \ell(\gamma) = k$. It suffices to show there is some $(r,s)$ and some $i$ such that $\gamma(r)=i$ and $\gamma(s) = i+1$ and $(r,s)$ is a descent pair in $\rho$, since then $\gamma\preceq_L s_i\gamma \preceq_L \rho$ by induction. Suppose for a contradiction that for every $(r,s)$ such that $\gamma(s)-\gamma(r) = 1$, we have $\rho(r) < \rho(s)$. Since every descent pair in $\gamma$ is necessarily also a descent pair in $\rho$, we also have that if $\gamma(t) = j+1$ and $\gamma(u) = j$ for $t<u$ then $\rho(t) > \rho(u)$. This implies that $\rho(\gamma^{-1}(1)) < \rho(\gamma^{-1}(2)) < \cdots < \rho(\gamma^{-1}(n))$, and therefore $\rho = \gamma$, contradicting $\ell(\rho) - \ell(\gamma) = k\ge 1$. 
\end{proof}

\begin{theorem}
Let $[\sigma, \rho]_L$ be an interval in left weak Bruhat order on $\mathfrak{S}_n$. Let $D$ and $\StdTab(D)$ be chosen such that $\{\rw(T) : T\in \StdTab(D)\}$ is equal to $[\sigma, \rho]_L$. Then $\overline{B}(\sigma, \rho)$ is isomorphic to ${\bf N}_{\StdTab(D)}$.
\end{theorem}
\begin{proof}
Lemma~\ref{lem:Bruhatascent} ensures ${\bf N}_{\StdTab(D)}$ is an $H_n(0)$-module. The statement then follows immediately from identifying an element of $\StdTab(D)$ with its reading word and comparing (\ref{eqn:barB}) and (\ref{eqn:0Hecke}).
\end{proof}

There are, however, ascent-compatible sets $\StdTab(D)$ whose reading words do not form an interval in left weak Bruhat order, and whose corresponding $H_n(0)$-module ${\bf N}_{\StdTab(D)}$ is moreover not isomorphic to any weak Bruhat interval module. 
\begin{example}
The set $\StdTab(D)$ of tableaux in Example~\ref{ex:0H} is ascent-compatible, but their reading words $\{213, 123, 132\}$ do not form an interval in left weak Bruhat order on $\mathfrak{S}_3$. Moreover, any interval in left weak Bruhat order with $3$ elements is a chain, and so elements of the corresponding weak Bruhat interval module can have at most one nonattacking ascent. On the other hand, the tableau $S\in \StdTab(D)$ (see Example~\ref{ex:0H}) has two nonattacking ascents, so ${\bf N}_{\StdTab(D)}$ is not isomorphic to any weak Bruhat interval module. 
\end{example}

The modules ${\bf N}_{\StdTab(D)}$ therefore generalize $\overline{B}(\sigma, \rho)$. 

\begin{remark}
The other family $B(\sigma, \rho)$ of weak Bruhat interval modules are isomorphic to $\widehat{{\bf N}}_{\StdTab(D)}$ for $D$ and $\StdTab(D)$ chosen such that $\{\rw(T) : T\in \StdTab(D)\} = \{\gamma w_0 : \gamma \in [\sigma, \rho]_L\}$ (for $w_0$ the longest permutation in $\mathfrak{S}_n$); in other words, the \emph{reverses} (in one-line notation) of the elements of $[\sigma, \rho]_L$. The isomorphism follows since the $H_n(0)$-action for $B(\sigma, \rho)$ is the same as (\ref{eqn:hat0Hecke}) up to a reversal of the role of ascents and descents, and $i$ is a descent of $\gamma \in [\sigma, \rho]_L$ if and only if $i$ is an ascent of $\gamma w_0$.
\end{remark}

\section{Future directions}

We conclude with some directions for future work on diagram modules. 

\subsection{Structure of diagram modules} It would be interesting to better determine structural properties of the diagram modules. There are many possible avenues here, for example 
\begin{enumerate}
\item determining under what conditions two diagram modules are isomorphic
\item determining under what conditions diagram modules are tableau-cyclic
\item finding restriction formulas for diagram modules
\item determining projective covers of diagram modules.
\end{enumerate}
Similar questions are asked or addressed for weak Bruhat interval $0$-Hecke modules in \cite{JKLO}.

\subsection{Applications of diagram modules} Diagram modules provide a method for constructing new families of functions in $\PQSym$, and realizing certain known families of functions in $\PQSym$, as peak characteristics of $HCl_n(0)$-supermodules. It would be interesting to see what other structures in $\PQSym$ can be obtained via diagram modules, to understand their properties, and in particular, to apply the diagram modules framework to construct analogues in $\PQSym$ of important bases in $\QSym$, along the lines of the quasisymmetric Schur $Q$-functions and the peak Young quasisymmetric Schur functions. The following questions are natural.

\begin{enumerate}
\item What other families of functions in $\PQSym$, new or known, can be constructed using the diagram modules framework?
\item For new families of functions in $\PQSym$ constructed in this way, how are they related to one another, and to important families of functions in $\Sym$, $\QSym$ or $\PQSym$, and what properties do they have? 
For example, the following questions, among others, arise for peak Young quasisymmetric Schur functions. Can we understand their multiplicative structure, e.g., by determining Pieri rules? Do quasisymmetric Schur $Q$-functions expand positively in peak Young quasisymmetric Schur functions? What can be said about their dual basis in the Hopf algebra dual of $\PQSym$ inside the noncommutative symmetric functions? Similar questions would also arise for other families in $\PQSym$ introduced using these methods.
\item For peak algebra analogues of families of functions in $\QSym$ constructed via diagram modules, how do their properties and their relationships to other families in $\PQSym$ compare with those for the analogous functions in $\QSym$? For example, the peak Young quasisymmetric Schur functions contain the Schur $Q$-functions in $\PQSym$, whereas the Young quasisymmetric Schur functions refine the Schur functions in $\QSym$.
\end{enumerate}

%
%

\section*{Acknowledgements}
The author was supported by the Marsden Fund, administered by the Royal Society of New Zealand Te Ap{\= a}rangi.

\bibliographystyle{abbrv} 
\bibliography{0HeckeClifford}

\end{document}